\documentclass{elsarticle}
\usepackage{a4wide}
\usepackage[utf8x]{inputenc}
\usepackage[english]{babel}
\usepackage[T1]{fontenc}
\usepackage{enumitem}

\usepackage{mathrsfs,stmaryrd}

\usepackage{amsmath,amsfonts,amssymb,amsthm}

\usepackage{tikz}
\usetikzlibrary{fit}
\usetikzlibrary{decorations.pathreplacing}
\usetikzlibrary{calc}

\usepackage{tikz-cd}

\usepackage{hyperref}
\newtheorem{question}{Question}

\theoremstyle{definition}
\newtheorem{definition}{Definition}[section]
\newtheorem{proposition}[definition]{Proposition}
\newtheorem{example}[definition]{Example}
\newtheorem{examples}[definition]{Examples}

\newtheorem{theorem}[definition]{Theorem}
\newtheorem{corollary}[definition]{Corollary}
\newtheorem{conjecture}[definition]{Conjecture}

\theoremstyle{remark}
\newtheorem{remark}{Remark}

\definecolor{ColWhite}{RGB}{255,255,255} 
\definecolor{labriBlue}{RGB}{0,114,201}
\definecolor{labriGreen}{RGB}{10,177,134}
\definecolor{labriTopRed}{RGB}{152,52,54}
\definecolor{grayVertex}{RGB}{210,210,210}

\tikzstyle{Centering}=[{baseline={([yshift=-0.5ex]current bounding box.center)}}]					
\tikzstyle{NodeGraph}=[circle,draw=black,inner sep=1pt,minimum size=4mm, thick,font=\scriptsize]	
\tikzstyle{UnlabeledNodeGraph}=[NodeGraph,minimum size=2mm]										
\tikzstyle{RootGraph}=[NodeGraph,rectangle]														
\tikzstyle{EdgeGraph}=[labriBlue,cap=round,very thick]												
\tikzstyle{ArcGraph}=[EdgeGraph,->]																	%
\tikzstyle{NullNode}=[font=\normalsize]
\tikzstyle{BlackNode}=[circle,draw=black,fill=black,inner sep=1pt,minimum size=3mm, thick,font=\scriptsize]
\tikzstyle{GrayNode}=[circle,line width=0.5,draw=black,fill=grayVertex,inner sep=1pt,minimum size=3mm,font=\scriptsize]
\tikzstyle{WhiteNode}=[circle,line width=0.5,draw=black,inner sep=1pt,minimum size=2mm,font=\scriptsize]
\tikzstyle{SmallBlackNode}=[circle,draw=black,fill=black,inner sep=1pt,minimum size=1mm, thick,font=\scriptsize]
\tikzstyle{SmallGrayNode}=[circle,line width=0.2,draw=black,fill=grayVertex,inner sep=1pt,minimum size=1.5mm,font=\scriptsize]
\tikzstyle{SmallGrayNodeii}=[circle,line width=0.2,draw=black,fill=grayVertex,inner sep=1pt,minimum size=1.2mm,font=\scriptsize]
\tikzstyle{MedGrayNode}=[circle,line width=0.5,draw=black,fill=grayVertex,inner sep=1pt,minimum size=2mm,font=\scriptsize]
\tikzstyle{SmallWhiteNode}=[circle,line width=0.2,draw=black,,inner sep=1pt,minimum size=1mm,,font=\scriptsize]
\tikzstyle{WhiteLabelNode}=[circle,fill=white,draw=black,inner sep=1pt,minimum size=5mm, thick,font=\scriptsize]
\tikzstyle{SmallWhiteLabelNode}=[circle,draw=black,inner sep=1pt,minimum size=3.5mm, thick,font=\scriptsize]

\tikzstyle{BlackPointNode}=[circle,fill=black,draw=black,inner sep=0pt,minimum size=1mm, thick,font=\scriptsize]

\tikzstyle{GrayPointNode}=[circle,fill=gray,draw=gray,inner sep=0pt,minimum size=1mm, thick,font=\scriptsize]

\usetikzlibrary{arrows,positioning}
 \tikzset{
    >=stealth,
    font=\smallsize,
    possible world/.style={circle,draw,thick,align=center},
    real world/.style={double,circle,draw,thick,align=center},
    minimum size=40pt
}
\tikzstyle{vertex}=[circle, draw, inner sep=0pt, minimum size=2pt]
\tikzset{
  pics/carc/.style args={#1:#2:#3}{
    code={
      \draw[pic actions] (#1:#3) arc(#1:#2:#3);
    }
  }
}

\newcommand{\GrayV}{\hspace{1pt}\begin{tikzpicture}[baseline=-2]\node[SmallGrayNode] at (0,0){};\end{tikzpicture}\hspace{1pt}}
\newcommand{\GrayVii}{\hspace{1pt}\begin{tikzpicture}[baseline=-2]\node[SmallGrayNodeii] at (0,0){};\end{tikzpicture}\hspace{1pt}}

\newcommand{\BlackV}{\hspace{1pt}\begin{tikzpicture}[baseline=-2]\node[SmallBlackNode] at (0,0){};\end{tikzpicture}\hspace{1pt}}

\definecolor{lightlightgray}{rgb}{0.85,0.85,0.85}

\newcommand{\modifJ}[2]{#2}  
  
\newcommand{\modifT}[2]{#2}

\newcommand{\OEIS}[1]{\href{http://oeis.org/#1}{{#1}}}              

\newcommand{\set}[1]{\left\{#1\right\}}                                 
\newcommand{\card}[1]{\left|#1\right|}            	                                

\newcommand{\mat}[1]{\mathbf{#1}}										
\newcommand{\perm}{\text{perm}}

\newcommand{\Cyc}{\mathfrak{C}}											

\begin{document}

\begin{frontmatter}

\title{Realizable cycle structures in digraphs}

\author[ULCO]{Jean Fromentin}

\author[ULCO]{Pierre-Louis Giscard\corref{corr}} 
\cortext[corr]{Corresponding author}
\ead{giscard@univ-littoral.fr}

\author[ULCO]{Th\'eo Karaboghossian}

\address[ULCO]{Univ. Littoral C\^{o}te d'Opale, UR 2597, LMPA, Laboratoire de Math\'ematiques Pures et Appliqu\'ees Joseph Liouville, F-62100 Calais, France}

\begin{abstract}
Simple cycles on a digraph form a trace monoid under the rule that two such cycles commute if and only if they are vertex disjoint. This rule describes the spatial configuration of simple cycles on the digraph. Cartier and Foata have showed that all combinatorial properties of closed walks are dictated by this trace monoid. We find that most graph properties can be lost while maintaining the monoidal structure of simple cycles and thus cannot be inferred from it, including vertex-transitivity, regularity, planarity, Hamiltonicity, graph spectra, degree distribution and more. 
Conversely we find that even allowing for multidigraphs, many configurations of simple cycles are not possible at all. The problem of determining whether a given configuration of simple cycles is realizable is highly non-trivial. We show at least that it is decidable and equivalent to the existence of integer solutions to systems of polynomial equations. 
\end{abstract}

\begin{keyword}
 walk\sep simple cycle\sep  trace monoid\sep hike monoid
\MSC[2010] 05C20\sep 05C38\sep 20M05
\end{keyword}

\end{frontmatter}

\section{Introduction}
The precise nature of the relation between graphs and their walks has, to the best of our knowledge, not yet been thoroughly scrutinized. It seems to be often assumed that walks are slave objects to graphs, in that once the graph is specified its walks can be studied and, in principle, be perfectly known. As a corollary it is expected that properties of the graph leave indelible imprints on its walks, from which the former can thus be inferred. These simple arguments mask the subtle nature of the relation between walks and graphs. For example, consider the following bidirected, vertex-transitive, hence regular, bipartite graph 
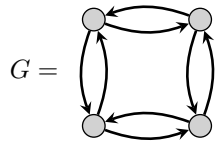
\begin{equation*}
\begin{tikzpicture}[line width=1,Centering]
\def\radi{1}
\node[NullNode] at (-1.5,0.05){$G=$};
\node[GrayNode](a) at (45:\radi){};
\node[GrayNode](b) at (135:\radi){};
\node[GrayNode](c) at (225:\radi){};
\node[GrayNode](d) at (315:\radi){};
\path[->](a) edge[bend right=20] (b);
\path[->](b) edge[bend right=20] (c);
\path[->](c) edge[bend right=20] (d);
\path[->](d) edge[bend right=20] (a);
\path[->](a) edge[bend right=20] (d);
\path[->](b) edge[bend right=20] (a);
\path[->](c) edge[bend right=20] (b);
\path[->](d) edge[bend right=20] (c);
\end{tikzpicture}
\end{equation*}
\noindent and let $W_{G:\GrayVii\to\GrayVii'}$ be the set of all walks from any vertex $\GrayV$ to any vertex $\GrayV'$ on $G$. Now consider the following two digraphs 
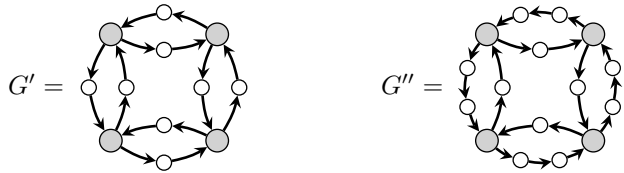
\begin{equation*}
\def\radi{1}
\def\radii{1}
\def\radiii{0.5}
\begin{tikzpicture}[line width=1,Centering]
\node[NullNode] at (-1.7,0.05){$G'=$};
\node[GrayNode](a) at (45:\radi){};
\node[GrayNode](b) at (135:\radi){};
\node[GrayNode](c) at (225:\radi){};
\node[GrayNode](d) at (315:\radi){};
\node[WhiteNode](a1) at (90:\radii){};
\node[WhiteNode](a2) at (90:\radiii){};
\node[WhiteNode](b1) at (180:\radii){};
\node[WhiteNode](b2) at (180:\radiii){};
\node[WhiteNode](c1) at (270:\radii){};
\node[WhiteNode](c2) at (270:\radiii){};
\node[WhiteNode](d1) at (0:\radii){};
\node[WhiteNode](d2) at (0:\radiii){};
\path[->] (a) edge [bend right=10] (a1);
\path[->] (b) edge [bend right=10] (a2);
\path[->] (a2) edge [bend right=10] (a);
\path[->] (a1) edge [bend right=10] (b);
\path[->] (b) edge [bend right=10] (b1);
\path[->] (c) edge [bend right=10] (b2);
\path[->] (b2) edge [bend right=10] (b);
\path[->] (b1) edge [bend right=10] (c);
\path[->] (c) edge [bend right=10] (c1);
\path[->] (d) edge [bend right=10] (c2);
\path[->] (c2) edge [bend right=10] (c);
\path[->] (c1) edge [bend right=10] (d);
\path[->] (d) edge [bend right=10] (d1);
\path[->] (a) edge [bend right=10] (d2);
\path[->] (d2) edge [bend right=10] (d);
\path[->] (d1) edge [bend right=10] (a);
\begin{scope}[shift={(5,0)}]
\node[NullNode] at (-1.7,0.05){$G''=$};
\node[GrayNode](a) at (45:\radi){};
\node[GrayNode](b) at (135:\radi){};
\node[GrayNode](c) at (225:\radi){};
\node[GrayNode](d) at (315:\radi){};
\node[WhiteNode](a1) at (75:\radii){};
\node[WhiteNode](a3) at (105:\radii){};
\node[WhiteNode](a2) at (90:\radiii){};
\node[WhiteNode](b1) at (165:\radii){};
\node[WhiteNode](b3) at (195:\radii){};
\node[WhiteNode](b2) at (180:\radiii){};
\node[WhiteNode](c1) at (255:\radii){};
\node[WhiteNode](c3) at (285:\radii){};
\node[WhiteNode](c2) at (270:\radiii){};
\node[WhiteNode](d1) at (345:\radii){};
\node[WhiteNode](d3) at (15:\radii){};
\node[WhiteNode](d2) at (0:\radiii){};
\path[->] (a) edge [bend right=10] (a1);
\path[->] (b) edge [bend right=10] (a2);
\path[->] (a2) edge [bend right=10] (a);
\path[->] (a1) edge [bend right=10] (a3);
\path[->] (a3) edge [bend right=10] (b);
\path[->] (b) edge [bend right=10] (b1);
\path[->] (c) edge [bend right=10] (b2);
\path[->] (b2) edge [bend right=10] (b);
\path[->] (b1) edge [bend right=10] (b3);
\path[->] (b3) edge [bend right=10] (c);
\path[->] (c) edge [bend right=10] (c1);
\path[->] (d) edge [bend right=10] (c2);
\path[->] (c2) edge [bend right=10] (c);
\path[->] (c1) edge [bend right=10] (c3);
\path[->] (c3) edge [bend right=10] (d);
\path[->] (d) edge [bend right=10] (d1);
\path[->] (a) edge [bend right=10] (d2);
\path[->] (d2) edge [bend right=10] (d);
\path[->] (d1) edge [bend right=10] (d3);
\path[->] (d3) edge [bend right=10] (a);
\end{scope}
\end{tikzpicture}
\end{equation*}
\noindent Remark that $G'$ is directed and neither vertex-transitive nor regular. Yet, the set of all walks from any gray vertex to any gray vertex on $G'$ is in bijection with the corresponding set $W_{G:\GrayV\to\GrayV'}$ on $G$. 
The correspondence between both walk sets is simple: the white transient vertices of $G'$ do not add any new cycle to the digraph as compared to $G$ and so only cause the lengths of all of these walks to be multiplied by~$2$. The structure of every individual walk, i.e., the way it is composed of cycles and of at most one open simple path, is clearly preserved. Rigorously, this makes the bijection into an isomorphism between monoids on walks and walk-like objects, called hikes, something we discuss in details later. 
Similarly, sets of walks between gray vertices on $G''$ are in bijection with the corresponding sets on $G$. In this case however the bijection does not act plainly on the length of individual walks but it continues to be a monoid isomorphism preserving the internal structure of walks. Remarkably here, $G''$ is not even bipartite. That is, from the point of view of graph theory the transformations from $G$ to $G'$ and from $G$ to $G''$ are nontrivial and subtle. 
They correspond to the loss of respectively three and four fundamental graph properties namely bidirectedness, vertex-transivity, regularity and bipartiteness. Yet walks sets are essentially unchanged by these alterations. 

At the very least, these observations suggest a rather loose relationship between graph properties and walk properties. It raises the question of whether there are more graph transformations--excluding graph automorphisms--that preserve walk sets, and if so are they all somewhat trivial, as above with the addition of transient vertices?  
{This question indirectly sprung up in the fields of applied network analysis and machine learning. For example it was noted that vertex centralities, mathematical quantities designed to grasp the relative importance of nodes in a graph, sometimes fail to do so meaningfully \cite{Qi2012,Opsahl2010}.} They can for {instance} predict that an outlying vertex is just as \modifJ{`central'}{``central''} as another one which appears to be at the heart of a graph: 
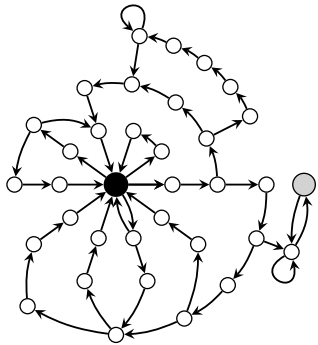
\begin{equation*}
\begin{tikzpicture}[line width=0.7,Centering]
\def\radi{0.75}
\def\radii{1.35}
\def\radiii{2}
\def\radiv{2.5}
\node[BlackNode](v8) at (0,0){};
\node[WhiteNode](v9) at (0:\radi){};
\node[WhiteNode](v18) at (108:\radi){};
\node[WhiteNode](v19) at (144:\radi){};
\node[WhiteNode](v22) at (180:\radi){};
\node[WhiteNode](v23) at (36:\radi){};
\node[WhiteNode](v24) at (72:\radi){};
\node[WhiteNode](v32) at (216:\radi){};
\node[WhiteNode](v29) at (252:\radi){};
\node[WhiteNode](v25) at (288:\radi){};
\node[WhiteNode](v7) at (324:\radi){};
\node[WhiteNode](v10) at (0:\radii){};
\node[WhiteNode](v17) at (108:\radii){};
\node[WhiteNode](v20) at (144:\radii){};
\node[WhiteNode](v21) at (180:\radii){};
\node[WhiteNode](v31) at (216:\radii){};
\node[WhiteNode](v28) at (252:\radii){};
\node[WhiteNode](v26) at (288:\radii){};
\node[WhiteNode](v6) at (324:\radii){};
\node[WhiteNode](v27) at (270:\radiii){};
\node[WhiteNode](v30) at (234:\radiii){};
\node[WhiteNode](v5) at (-63:\radiii){};
\node[WhiteNode](v11) at (0:\radiii){};
\node[WhiteNode](v3) at (-21:\radiii){};
\node[WhiteNode](v4) at (-42:\radiii){};
\node[WhiteNode](v12) at (27:\radii){};
\node[WhiteNode](v15) at (54:\radii){};
\node[WhiteNode](v16) at (81:\radii){};
\node[WhiteNode](v13) at (27:\radiii){};
\node[WhiteNode](v14) at (40.5:\radiii){};
\node[WhiteNode](v33) at (54:\radiii){};
\node[WhiteNode](v34) at (67.5:\radiii){};
\node[WhiteNode](v35) at (81:\radiii){};
\node[WhiteNode](v2) at (-21:\radiv){};
\node[GrayNode](v1) at (0:\radiv){};
\draw[->](v8) to (v9);
\draw[->](v18) to (v8);
\draw[->](v8) to (v19);
\draw[->](v22) to (v8);
\draw[->](v8) to (v23);
\draw[->](v24) to (v8);
\draw[->](v32) to (v8);
\draw[->](v29) to (v8);
\path[->](v25) edge[bend left=15] (v8);
\path[->](v8) edge[bend left=15] (v25);
\draw[->](v7) to (v8);
\draw[->](v8) to (v9);
\path[->](v23) edge[bend right=10] (v24);
\draw[->](v19) to (v20);
\draw[->](v21) to (v22);
\path[->](v20) edge[bend right=10] (v21);
\path[->](v20) edge[bend left=20] (v18);
\path[->](v9) edge (v10);
\path[->](v17) edge[] (v18);
\draw[->](v31) to (v32);
\draw[->](v28) to (v29);
\draw[->](v25) to (v26);
\path[->](v6) edge (v7);
\path[->](v26) edge[bend left=10] (v27);
\path[->](v27) edge[bend left=10] (v28);
\path[->](v27) edge[bend left=10] (v30);
\path[->](v30) edge[bend left=10] (v31);
\path[->](v5) edge[bend right=10] (v6);
\path[->](v5) edge[bend left=10] (v27);
\path[->](v10) edge (v11);
\path[->](v11) edge[bend left=10] (v3);
\path[->](v3) edge[bend left=10] (v4);
\path[->](v4) edge[bend left=10] (v5);
\path[->](v16) edge[bend right=10] (v17);
\path[->](v10) edge[bend right=10] (v12);
\path[->](v12) edge[bend right=10] (v15);
\path[->](v15) edge[bend right=10] (v16);
\draw[->](v12) to (v13);
\draw[->](v35) to (v16);
\draw[->](v3) to (v2);
\path[->](v2) edge[bend right=20] (v1);
\path[->](v1) edge[bend right=10] (v2);
\path[->](v13) edge[bend right=10] (v14);
\path[->](v14) edge[bend right=10] (v33);
\path[->](v33) edge[bend right=10] (v34);
\path[->](v34) edge[bend right=10] (v35);
\path[->](v35) edge [out=70,in=140,distance=15] (v35);
\path[->](v2) edge [out=210,in=280,distance=15] (v2);
\end{tikzpicture}
\end{equation*}

\noindent On this digraph with adjacency matrix $\mat{A}$, the gray and black vertices are deemed to be equally central by the subgraph centrality measure {$(e^{\mat{A}})_{\GrayVii}= (e^{\mat{A}})_{\BlackV}$}, a widely used index of vertex importance in real-world networks {\cite{Estrada2}}.
Similarly, measures aimed at grasping the degree of {``similarity''} between networks are prone to failures. {In the field of machine learning these measures, called graph-kernels, are used primarily for tasks of automatic graph classification. In particular it is known that walk-based graph-kernels suffer from the existence of differing graph structures with similar walk counts or arrangement of walks, hindering the automatic distinction of these graph structures, see  e.g. \cite{ramon03graphkernels, Mahe2009, Kriege2020} and references therein. To illustrate this observation, consider the following} four directed graphs 
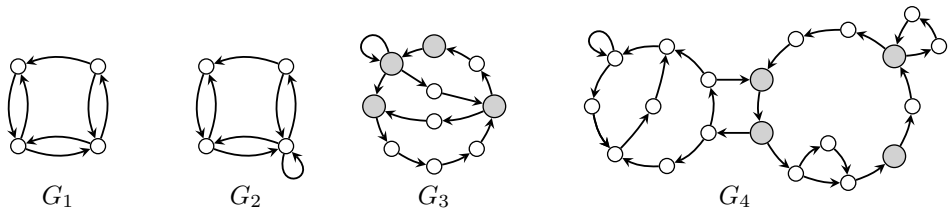
\begin{equation*}
\begin{tikzpicture}[line width=0.7,Centering]
\def\radi{0.75}
\def\radii{1}
\node[WhiteNode](a1) at (45:\radi){};
\node[WhiteNode](a2) at (135:\radi){};
\node[WhiteNode](a3) at (225:\radi){};
\node[WhiteNode](a4) at (315:\radi){};
\path[->](a1) edge[bend right=20] (a2);
\path[->](a2) edge[bend right=20] (a3);
\path[->](a3) edge[bend right=20] (a4);
\path[->](a4) edge[bend right=20] (a1);
\path[->](a3) edge[bend right=20] (a2);
\path[->](a4) edge[bend right=20] (a3);
\path[->](a1) edge[bend right=20] (a4);
\node[NullNode] at (0,-1.2){$G_1$};
\begin{scope}[shift={(2.5,0)}]
\node[WhiteNode](b1) at (135:\radi){};
\node[WhiteNode](b2) at (225:\radi){};
\node[WhiteNode](b3) at (315:\radi){};
\node[WhiteNode](b4) at (45:\radi){};
\path[->](b1) edge[bend right=20] (b2);
\path[->](b2) edge[bend right=20] (b1);
\path[->](b2) edge[bend right=20] (b3);
\path[->](b3) edge[bend right=20] (b2);
\path[->](b3) edge[bend right=20] (b4);
\path[->](b4) edge[bend right=20] (b3);
\path[->](b4) edge[bend right=20] (b1);
\path[->](b3) edge [out=255,in=315,distance=15] (b3);
\node[NullNode] at (0,-1.2){$G_2$};
\end{scope}
\def\radi{0.2}
\def\radii{0.8}
\begin{scope}[shift={(5,0)}]
\node[GrayNode](c1) at (0:\radii){};
\node[WhiteNode](c2) at (45:\radii){};
\node[GrayNode](c3) at (90:\radii){};
\node[GrayNode](c4) at (135:\radii){};
\node[GrayNode](c5) at (180:\radii){};
\node[WhiteNode](c6) at (225:\radii){};
\node[WhiteNode](c7) at (270:\radii){};
\node[WhiteNode](c8) at (315:\radii){};
\node[WhiteNode](c9) at (90:\radi){};
\node[WhiteNode](c10) at (270:\radi){};
\path[->](c1) edge[bend right=10] (c2);
\path[->](c2) edge[bend right=10] (c3);
\path[->](c3) edge[bend right=10] (c4);
\path[->](c4) edge[bend right=10] (c5);
\path[->](c5) edge[bend right=10] (c6);
\path[->](c6) edge[bend right=10] (c7);
\path[->](c7) edge[bend right=10] (c8);
\path[->](c8) edge[bend right=10] (c1);
\path[->](c4) edge[bend left=0] (c9);
\path[->](c9) edge[bend left=0] (c1);
\path[<-](c5) edge[bend right=10] (c10);
\path[<-](c10) edge[bend right=10] (c1);
\path[->](c4) edge[out=105,in=165,distance=15] (c4);
\node[NullNode] at (0,-1.2){$G_3$};
\end{scope}
\begin{scope}[shift={(9,0)}]
\def\radi{0.5}
\node[GrayNode](d1) at (45:\radi){};
\node[WhiteNode](d2) at (135:\radi){};
\node[WhiteNode](d3) at (225:\radi){};
\node[GrayNode](d4) at (315:\radi){};
\begin{scope}[shift={(-2.175*\radi,0)}]
\node[WhiteNode](d5) at (77.143:1.629*\radi){};
\node[WhiteNode](d6) at (128.571:1.629*\radi){};
\node[WhiteNode](d7) at (180:1.629*\radi){};
\node[WhiteNode](d8) at (231.428:1.629*\radi){};
\node[WhiteNode](d9) at (282.857:1.629*\radi){};
\node[WhiteNode](d17) at (0,0){};
\end{scope}
\begin{scope}[shift={(2.65*\radi,0)}]
\node[WhiteNode](d10) at (240:2.067*\radi){};
\node[WhiteNode](d11) at (280:2.067*\radi){};
\node[GrayNode](d12) at (320:2.067*\radi){};
\node[WhiteNode](d13) at (0:2.067*\radi){};
\node[GrayNode](d14) at (40:2.067*\radi){};
\node[WhiteNode](d15) at (80:2.067*\radi){};
\node[WhiteNode](d16) at (120:2.067*\radi){};
\node[WhiteNode](d18) at (30:3.2*\radi){};
\node[WhiteNode](d19) at (50:3.2*\radi){};
\node[WhiteNode](d20) at (260:\radi){};
\end{scope}
\path[->](d1) edge[bend right=10] (d4);
\path[->](d4) edge (d3);
\path[->](d3) edge[bend right=15] (d2);
\path[->](d2) edge (d1);
\path[->](d2) edge[bend right=15] (d5);
\path[->](d5) edge[bend right=15] (d6);
\path[->](d6) edge[bend right=15] (d7);
\path[->](d7) edge[bend right=15] (d8);
\path[->](d7) edge[bend right=15] (d8);
\path[->](d9) edge[bend left=15] (d8);
\path[->](d3) edge[bend left=15] (d9);
\path[->](d4) edge[bend right=10] (d10);
\path[->](d10) edge[bend right=10] (d11);
\path[->](d11) edge[bend right=10] (d12);
\path[->](d12) edge[bend right=10] (d13);
\path[->](d13) edge[bend right=10] (d14);
\path[->](d14) edge[bend right=10] (d15);
\path[->](d15) edge[bend right=10] (d16);
\path[->](d16) edge[bend right=10] (d1);
\path[->](d8) edge[bend right=5] (d17);
\path[->](d17) edge[bend right=5] (d5);
\path[->](d6) edge[out=100,in=160,distance=15] (d6);
\path[->](d14) edge[bend right=10] (d18);
\path[->](d18) edge[bend right=20] (d19);
\path[->](d19) edge[bend right=10] (d14);
\path[->](d10) edge[bend left=20] (d20);
\path[->](d20) edge[bend left=20] (d11);
\path[->](d4) to (d3);
\path[->](d3) to (d2);
\path[->](d2) to (d1);
\node[NullNode] at (0,-1.2){$G_4$};
\end{scope}
\end{tikzpicture}
\end{equation*}
\noindent While digraphs $G_1$ and $G_2$ might seem most similar with one another among all four, sets of walks between any pairs of gray vertices on graphs $G_3$ and $G_4$ are in bijection with sets of walks on~$G_1$. As in the earlier example, this bijection preserves the structure of individual walks. 
At the opposite, there is no such correspondence between the set of all walks on graph $G_2$ and those on $G_1$.\\[-.5em]

When questioned, these failures have at times been put down to the underlying non-rigorous notions of ``importance of a node'' and of ``similarity between graphs'' as we \textit{intuitively} understand them. It is argued that such an understanding, imprecise and fraught with preconceptions, is difficult to express rigorously in mathematical terms thus leading to seemingly absurd results. But the failures could also run deeper and be yet more manifestations of the misappreciated mathematical relation between  graphs  and walks. Indeed almost all of the measures proposed so far for quantifying such notions as network centrality and graph similarity are algebraic quantities which translate into statements about closed walks. If indeed the cause of the problem is that graphs and their walks give relatively little or at least indirect information about one another, then this issue should be made mathematically rigorous. Considering in particular the case of closed walks on strongly connected digraphs, is perfect knowledge of the former sufficient to determine fundamental properties of the later, possibly up to some triviality? 

\modifJ{In this work, we shall answer negatively to this question}{In this work, we answer the preceding question in the negative}: not only are there many transformations between strongly connected digraphs that profoundly alter their properties yet leave their sets of closed walks invariant, but these transformations are diverse and far from trivial. Some produce bijective mappings between walk sets that not only preserve their monoidal structure as in the above examples, but also the length of \textit{all} closed walks. Algebraic quantities routinely used to characterize graphs are left invariant under such mappings. A complete classification of all walk-preserving graph transformations seems to be particularly arduous. The problem finds a wider context in the assertion that a digraph can be reconstructed based solely on structural information about its simple cycles, something which, we find, fails to hold in many a strange way. The complementary question, namely deciding whether a digraph exhibiting certain structural relations between simple cycles exists at all turns out to be unexpectedly difficult. We show at least that this question is decidable and give a meaning for and examples of algebraically closed sets of walks and walk-like objects that cannot, by themselves, be drawn on digraphs. In one more unexpected result we find that these \modifJ{`undrawable'}{``undrawable''} walks do exist on larger digraphs where they are accompanied by a host of algebraically unrelated walks. 

In order to talk {about} closed walks and digraphs in a clearly separate manner so as to untangle their thorny relationship, it is necessary to have a mathematical language for describing sets of closed walks and walk-like objects independently from the digraphs that sustain them.
Such a language emerged in a series of contributions starting 
from the pioneering work by Cartier and Foata on partially 
commutative monoids \cite{cartier1969}, through Viennot's heaps 
of pieces theory 
\cite{viennot1989heaps,krattenthaler2006theory} to Giscard and 
Rochet's hike monoids \cite{GiscardRochet2016}. 
 All of the questions presented above then find their natural 
formulation as statements about trace and hike monoids. This is 
what we present below. {An alternative approach consisting in studying the syntactic structure of walk sets using language theory has also been developed \cite{Mosbah1997,Mosbah2000,Thwaite2014,Lindorfer2020}.}

\subsection{The language of closed walks: hike monoids}
\subsubsection{Notations for graphs and rooted walks}
While we begin by recalling standard definitions for graphs, we introduce somewhat less common concepts for walks, of which we advise the reader to take special notice.\\[-.5em]

 A \emph{graph} $G=(V,E)$ is a finite set of vertices $V$ and a finite set $E$ of distinct paired vertices, called edges, denoted $\{i,j\}$, $i,j\in V$. 
  A \emph{digraph} $G=(V,E)$ is a finite set of vertices $V$ and a finite set $E\subseteq V^2$ of  {\em directed edges} (or {\em arcs}), denoted $(i,j)$ for the arc from $i$ to $j$. 
 A \emph{directed multigraph} (or \emph{multidigraph}) is defined the same way as a digraph, except that $E$ is a multiset. {An edge of~$E$ is then denoted $(i,j)_k$, the integer $k$ specifying which edge from $i$ to $j$ we consider.} We denote by $\mat{A}$ the adjacency matrix of $G$ defined as $\mat{A}_{ij}:=n$ with $n\geq 0$ the number of directed edges $(i, j)_{k}\in E$ from vertex $i\in V$ to vertex $j\in V$. 
 
 
  We say that $G'=(V',E')$ is a {\em subgraph} of a multidigraph $G$, denoted by  $G'\subseteq G$, if $V'\subseteq V$ and $E'\subseteq E$. We say that $G'$ is 
an {\em induced subgraph} of a digraph $G$ if furthermore we have $E'=V'^2\cap E$, that is $G'$ is obtained from $G$ by deleting some of its vertices and only the edges adjacent to them.\\[-.5em] 

A \textit{rooted walk}, or rooted path, of length $\ell$ from vertex $i$ to vertex $j$ on a multi directed graph $G$ is a contiguous sequence of $\ell$ arcs starting from $i$ and ending in $j$, e.g. $w = (i, i_1)_{k_1}(i_1, i_2)_{k_2} \cdots (i_{\ell-1}, j)_{k_\ell}$ (a sequence of arcs is said to be contiguous if each arc but the first one starts where the previous ended). 
The rooted walk $w$ is \textit{open} if $i \neq j$ and \textit{closed} otherwise, in which case it is also called {\em rooted cycle}. A rooted cycle $(i_0, i_1)_{k_1}(i_1, i_2)_{k_2} \cdots (i_{\ell-1}, i_0)_{k_\ell}$ of non-zero length for which all vertices $i_t$ are distinct is said to be {\em simple}. A self-loop $(i, i)_k$ is considered a rooted simple cycle of length one. On digraphs we may also represent walks unambiguously as ordered sequences of vertices $w=i,i_1,\cdots ,i_{\ell-1},j$.

Discarding the piece of information regarding the position of the root turns rooted simple cycles into \emph{simple cycles}. 
A similar notion concerning rooted closed walks is presented in \ref{subsec:monoids}
An \textit{induced cycle} is a simple cycle for which no pair of visited vertices is linked by an arc that does not belong to the cycle.\\[-.5em] 

A {\em strongly connected component} of a digraph $G$ is a maximal  digraph 
$G'\subseteq G$ such that for every pair of vertices $v_1, v_2$ in $G'$, there is a rooted path in $G'$ from $v_1$ to $v_2$. A digraph $G$ is then said to be {\em strongly connected} if it is its sole strongly connected component.

\subsubsection{Trace monoids and hike monoids}\label{subsec:monoids}
Let $\Sigma$ be an alphabet then its Kleene star $\Sigma^*$ designates the set of all {finite} words on the letters of $\Sigma$. Now let $\mathcal{I}\subseteq \Sigma^2$ be a set of pairs of letters. This set defines a rule, called independence relation, which affirms that pairs of letters in $\mathcal{I}$ are independent and can be commuted when they occur next to each other in a word. Thus $\mathcal{I}$ induces an equivalence relation $\sim_{\mathcal{I}}$ between words $w_1,w_2\in \Sigma^*$ with $w_1\sim_\mathcal{I} w_2$ if an only if it is possible to pass from $w_1$ to $w_2$ by commuting adjacent pairs of independent letters. Then the \emph{trace monoid} $\mathcal{T}=\Sigma^*/\!\sim_\mathcal{I}$, is the free partially commutative monoid formed by the $\sim_{\mathcal{I}}$ equivalence classes on $\Sigma^*$. These classes are generically called \emph{traces}. At the heart of the trace monoid $\mathcal{T}$ is the partially commutative structure induced by $\mathcal{I}$. This structure is best represented as a graph, called the \textit{dependency graph} $H$ of $\mathcal{T}$. It is the graph for which vertices represent all letters of the alphabet $\Sigma$ and two vertices $v_i$ and $v_j$ are joined by an undirected edge if and only if letters $i,j\in\Sigma$ are \textit{not} allowed to commute per $\mathcal{I}$. Formally, $H$ is the complement graph of $(\Sigma, \mathcal{I})$.\\

Trace monoids where introduced by Pierre Cartier and Dominique Foata in their quest for a purely combinatorial proof of MacMahon's master theorem \cite{cartier1969,MacMahon1915}. They considered more specifically the trace monoid on the alphabet of labeled directed edges $\Sigma = E$ of a digraph $G=(V, E)$ with independence relation 
$$
\mathcal{I}_{\text{CF}} := \big\{\left\{(i, j),(k, l)\right\}:\, i\neq k\big\},
$$
where $(i,j)\in E$ and $(k,l)\in E$ are directed edges of $G$. This rule implies that two directed edges with different starting points are allowed to commute. For example, words $(1,2)(2,3)(1,4)$ and $(2,3)(1,2)(1,4)$ belong to the same Cartier-Foata trace, while $(2,3)(1,4)(1,2)$ belongs to a distinct trace due to the forbidden commutation of $(1,4)$ with $(1,2)$ required to pass from the former trace to the latter. This specific instance of trace monoid is sometimes called a \emph{Cartier-Foata monoid}.\\

Of particular importance for the rest of this study is the submonoid of a Cartier-Foata monoid formed by all traces for which the numbers of incoming and outgoing edges at each vertex are equal. This submonoid is a called a \textit{hike monoid} and will be denoted $\mathcal{H}$. It turns out \cite{GiscardRochet2016} that {hike monoids have a simpler presentation as partially commutative monoids on the alphabet $\mathcal{C}$ of directed graph simple cycles}, $\mathcal{H}\simeq \mathcal{C}^\ast/\mathcal{I}_{\text{Hike}}$, with the independence relation
$$
\mathcal{I}_{\text{Hike}} :=\big\{\left\{c_1,c_2\right\}:\,V(c_1)\cap V(c_2)=\emptyset\big\},
$$
where $V(c_i)$ denotes the set of vertices visited by cycle $c_i\in\mathcal{C}$ and $\mathcal{C}$ is the set of all simple cycles. This rule means that two simple cycles are allowed to commute if and only if they have no vertex in common. 
We emphasize that simple cycles $c\in\mathcal{C}$ are not rooted, rather they are considered up to cyclic permutations of their vertices, i.e., their starting point is irrelevant. This is because in the present context a \textit{simple cycle} really is a Cartier-Foata trace and by $\mathcal{I}_{\text{CF}}$ two words of $E^\ast$ where all vertices have one incoming and one outgoing edge and which differ only by their starting point are equivalent. 
For example, in a digraph $(1, 2)(2,3)(3,1)\sim_{\mathcal{I}_{\text{CF}}}(2,3)(3,1)(1,2)$: both words represent the same triangle. By contrast, we verify that the orientation of the simple cycles is preserved by the {Cartier-Foata independence relation $\mathcal{I}_{\text{CF}}$ since it does not allow reversing edge directions}. So for example $(1,2)(2,3)(3,1)$ and $(1,3)(3,2)(2,1)$ are distinct. 
The dependency graph of a hike monoid will be called the \emph{hike dependency graph}.

Elements of hike monoids are termed \textit{hikes} (Cartier and Foata used the French term ``circuits'' for the hikes \cite{cartier1969}, however this name  is now widely used in graph theory to designate other objects so \cite{GiscardRochet2016} adopted the term ``hikes'' to avoid confusion). Hikes are equivalent classes on words of simple cycles and are mathematically best understood  as heaps of such cycles on a graph. In this case, the simple cycles are pieces piled upon one another in such a way that two simple cycles can only be put at the same level if they share no vertex in common.
Hikes include \textit{closed walks} as special cases, more precisely it was shown that a hike $h=c_1\cdots c_k$ whose right-most simple cycle $c_k$ is unique is a walk $w$ in the sense that the words of $E^*$ belonging to the equivalence class $w$ are exactly all the rooted closed walks starting from the vertices of $c_k$ \cite{GiscardRochet2016}.
An arbitrary hike is a walk-like object (hence the name) appearing, when depicted graphically, to be a collection of closed walks. 
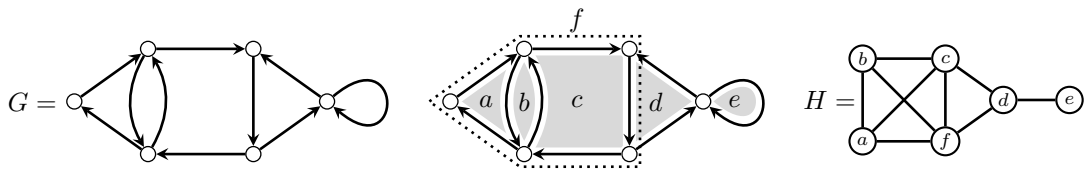
\begin{figure}[h!]
\vspace{-4mm}
    \centering
\begin{tikzpicture}[line width=1,Centering]
\def\radi{1.4}
\node[WhiteNode] (a) at (0,0){};
\node[WhiteNode] (b) at (0,\radi){};
\node[WhiteNode] (c) at (\radi,\radi){};
\node[WhiteNode] (d) at (\radi,0){};
\node[WhiteNode] (e) at (-0.7*\radi,0.5*\radi){};
\node[WhiteNode] (f) at (1.7*\radi,0.5*\radi){};
\draw[->] (b) to (c);
\draw[->] (c) to (d);
\draw[->] (d) to (a);
\draw[->] (a) to (e);
\draw[->] (e) to (b);
\draw[->] (d) to (f);
\draw[->] (f) to (c);
\path[->] (f) edge[out=40,in=-40,distance=35] (f);
\path[->] (a) edge[bend right=30] (b);
\path[->] (b) edge[bend right=30] (a);
\node[NullNode](l) at (-1.1*\radi,0.5*\radi){$G=$};
\begin{scope}[shift={(5,0)}]
\def\radi{1.4}
\def\radii{0.14}
\def\sep{5}
\draw[fill=lightlightgray](-0.7*\radi,0.5*\radi) -- (0,\radi) -- (\radi,\radi) -- (1.7*\radi,0.5*\radi) -- (\radi,0) -- (0,0);
\draw[fill=lightlightgray,color=lightlightgray](1.7*\radi,0.5*\radi) -- (2*\radi,0.3*\radi) -- (2.2*\radi,0.4*\radi) -- (2.2*\radi,0.6*\radi) -- (2*\radi,0.7*\radi);
\fill[color=white] (0,0) circle(\radii);
\fill[color=white] (0,\radi) circle(\radii);
\fill[color=white] (\radi,0) circle(\radii);
\fill[color=white] (\radi,\radi) circle(\radii);
\fill[color=white] (-0.7*\radi,0.5*\radi) circle(\radii);
\fill[color=white] (1.7*\radi,0.5*\radi) circle(\radii);
\node[WhiteNode] (a) at (0,0){};
\node[WhiteNode] (b) at (0,\radi){};
\node[WhiteNode] (c) at (\radi,\radi){};
\node[WhiteNode] (d) at (\radi,0){};
\node[WhiteNode] (e) at (-0.7*\radi,0.5*\radi){};
\node[WhiteNode] (f) at (1.7*\radi,0.5*\radi){};
\draw[color=white,line width=\sep] (b) to (c);
\draw[color=white,line width=\sep] (c) to (d);
\draw[color=white,line width=\sep] (d) to (a);
\draw[color=white,line width=\sep] (a) to (e);
\draw[color=white,line width=\sep] (e) to (b);
\draw[color=white,line width=\sep] (d) to (f);
\draw[color=white,line width=\sep] (f) to (c);
\path[color=white,line width=\sep] (a) edge[bend right=30] (b);
\path[color=white,line width=\sep] (b) edge[bend right=30] (a);
\path[color=white,line width=\sep] (f) edge[out=40,in=-40,distance=35] (f);
\draw[dotted](-0.9*\radi,0.5*\radi) -- (-0.04*\radi,1.12*\radi) -- (1.1*\radi,1.12*\radi) -- (1.1*\radi,-0.12*\radi) -- (-0.04*\radi,-0.12*\radi)  -- (-0.9*\radi,0.5*\radi);
\draw[->] (b) to (c);
\draw[->] (c) to (d);
\draw[->] (d) to (a);
\draw[->] (a) to (e);
\draw[->] (e) to (b);
\draw[->] (d) to (f);
\draw[->] (f) to (c);
\path[->] (f) edge[out=40,in=-40,distance=35] (f);
\path[->] (a) edge[bend right=30] (b);
\path[->] (b) edge[bend right=30] (a);
\node[NullNode] (l) at (2*\radi,0.5*\radi){$e$};
 \node[NullNode] (l) at (0.5*\radi,0.5*\radi){$c$};
\node[NullNode] (l) at (0,0.5*\radi){$b$};
\node[NullNode] (l) at (-0.36*\radi,0.5*\radi){$a$};
\node[NullNode] (l) at (1.25*\radi,0.5*\radi){$d$};
\node[NullNode] (l) at (0.5*\radi,1.27*\radi){$f$};
\end{scope}
\begin{scope}[shift={(9.5,0.17)}]
\def\radi{1.1}
\node[SmallWhiteLabelNode] (a) at (0,0){$a$};
\node[SmallWhiteLabelNode] (b) at (0,\radi){$b$};
\node[SmallWhiteLabelNode] (c) at (\radi,\radi){$c$};
\node[SmallWhiteLabelNode] (f) at (\radi,0){$f$};
\node[SmallWhiteLabelNode] (d) at (1.7*\radi,0.5*\radi){$d$};
\node[SmallWhiteLabelNode] (e) at (2.5*\radi,0.5*\radi){$e$};
\draw(a) to (b);
\draw(b) to (c);
\draw(a) to (c);
\draw(c) to (d);
\draw(d) to (f);
\draw(d) to (e);
\draw(a) to (f);
\draw(c) to (f);
\draw(b) to (f);
\node[NullNode](l) at (-0.4*\radi,0.5*\radi){$H=$};
\end{scope}
\end{tikzpicture}
\vspace{0mm}
\caption{An example of a digraph $G$ (left), the simple cycles on $G$ (middle) and the corresponding hike dependency graph $H$ (right).}
\end{figure}

\section{Preliminaries}\label{section_observations}
With the help of dependency graphs, hike and trace monoids provide a mean of discussing the properties of sets of walks and walk-like objects without explicitly referring to the graph on which these exist. Indeed, the dependency graph $H$ of a hike monoid $\mathcal{H}$ encodes the {``cycle skeleton''} of a \modifT{graph}{digraph} $G$, as vertices of $H$ are simple cycles on $G$ and edges in $H$ exist whenever two simple cycles of $G$ share at least one vertex. {In particular, sets of closed walks and more generally of hikes on two digraphs $G$ and $G'$ sharing the \textit{same} hike dependency graph $H$ are related through a bijection. This bijection stems from the isomorphism between their hike monoids $\mathcal{H}_{G}$ and $\mathcal{H}_{G'}$. Indeed, these monoids have the same presentation because they have the same number of generators (this being the number of nodes in the hike dependency graph $H$) as well as the same commutation relations between these generators (the edges in the complement $H^c$ of $H$). Then $\mathcal{H}_{G}$ and $\mathcal{H}_{G'}$ are isomorphic and there is a bijection between their sets of words, that is the sets of hikes (and thus of walks) on $G$ and $G'$}. Conversely, the existence of an isomorphism between hike monoids $\mathcal{H}_{G}$ and $\mathcal{H}_{G'}$ implies that $G$ and $G'$ share the same hike dependency graph. Relating graph properties and walk properties is thus equivalent to understanding which hike dependency graphs do exist and what properties of a digraph can we ascertain from its sole hike dependency graph.\\[-.5em]


Let $\phi$ be the map which
to any digraph $G$ associates $\phi(G)$, the dependency graph of the hike monoid of $G$. We say that two digraphs $G_1$ and $G_2$ are 
 $\phi$-equivalent if $\phi(G_1) = \phi(G_2)$ and we denote this by $G_1\sim_{\phi} G_2$. Any digraph transformation passing between two $\phi$-equivalent digraphs implements an isomorphism between their hike monoids since these two digraphs share the same hike dependency graph. 
The problem of understanding the relation between graph properties and walk properties can therefore be formulated as two questions on $\phi$:

\begin{question}[$\phi$-Surjectivity.] \label{question_surj}
	Given a graph $H$, is there a digraph $G$ such that $\phi(G) = H$? 
\end{question}

If this is the case we say that $H$ is {\em realizable} and that $G$ {\em realizes} $H$. Question~1 is equivalent to asking which trace monoid are hike monoids since any graph $H$ is necessarily the dependency graph of a trace monoid.

\begin{question}[$\phi$-Injectivity.] \label{question_inj}
	Given a graph $H$ determine $\phi^{-1}(H)$. Equivalently, what are all digraph transformations that induce isomorphisms of hike monoids\modifJ{ ?}{?}
\end{question}

From the examples presented earlier in the introduction we know already that $\phi$ is not injective. A good question is thus to characterize the class $\phi^{-1}(H)$ of all digraphs with hike dependency graph $H$. This, in turn, is equivalent to determining all digraphs transformations that induce isomorphisms of hike monoids since any two digraphs in $\phi^{-1}(H)$ are related by such a transformation.\\[-.5em]

Observe that hike monoids, map $\phi$ and questions~\ref{question_surj} and \ref{question_inj} naturally extend to 
multidigraphs. In fact the set of all graphs realizable by multidigraphs is exactly the set of  graphs realizable by digraphs since every digraph is a multidigraph and, 
for any directed multigraph, we can construct a digraph with identical hike monoid by adding transient vertices in the middle of each directed multi-edge. As a consequence, from now on we work equally with multidigraphs and digraphs referring to both as \modifJ{`digraphs'}{``digraphs''}.

{For $n$ a positive integer, denote by $[n]$ the set $\set{1,\hdots, n}$.}
Before we address the surjectivity and injectivity questions, we begin with three properties of the map $\phi$ that reduce considerations to connected dependency graphs $H$ and strongly connected multidigraphs $G$:

\begin{proposition}{Let $G$ be a digraph. We have}
	\begin{enumerate}
		\item If $G$ is strongly connected, then its hike dependency graph $\phi(G)$ is a connected graph.
		\item Let $G_1,\dots, G_n$ be the strongly connected components of $G$. Then $\phi(G) = \bigsqcup_{i\in {[n]}} \phi(G_i)$.
		\item Let $H$ be a connected realizable graph. Then $H$ is realizable by a strongly connected digraph.
	\end{enumerate}
\end{proposition}

\begin{proof}~
	\begin{enumerate}
		\item Let $c$ and $c'$ be two simple cycles of $G$ and $v$ and $v'$ respectively be a vertex in $c$ and a vertex in $c'$. Since $G$ is strongly connected,
		there exists a cycle containing the vertices $v$ and $v'$. Hence there exists a sequence of simple cycles $c_1,\dots, c_k$ such that each pairs of simple
		cycles $\set{c,c_1}, \set{c_1,c_2},\dots, \set{c_{n-1}, c_n}$ and $\set{c_n,c'}$ share a vertex in common. This implies that the vertices in $\phi(G)
		$ corresponding to $c$ and $c'$ are in the same connected component, and hence $\phi(G)$ is connected.
		\item By definition, if there is a simple cycle in $G_1$ then this simple cycle is also in $G$, and if two simple cycles of $G_1$ share a vertex in common,
		then they also share a vertex in common in $G$. Hence we have that $\phi(G_1)\subseteq \phi(G)$. Since this is also true for every $i\in{[n]}$, we have that
		$\bigsqcup_{i\in{[n]}} \phi(G_i) \subseteq \phi(G)$. Let $c$ be a simple cycle of $G$. Since for every pair of vertices $v_1,v_2$ in $c$ there is a path from
		$v_1$ to $v_2$ in $c$, $c$ is a directed subgraph of a connected component $G_i$ of $G$. Let now be $c'$ an other simple cycle of $G$ which share a vertex
		in common with $c$. Then $c'$ is also in $G_i$, since we can go from any vertex of $c$ to every vertex of $c'$ with a path in $c\sqcup c'$ and vice versa. 
		Hence $c$ and $c'$ are both simple cycles of $G_i$ and they also share a vertex in common in $G_i$. This implies that $\phi(G) \subseteq \bigsqcup_{i\in{[n]}}\phi(G_i)$
		and so $\phi(G) = \bigsqcup_{i\in {[n]}} \phi(G_i)$.
		\item Let $G$ be a digraph realising $H$ and  $G_1,\dots, G_n$ be its strongly connected components. By what precedes, we have that $H=\bigsqcup_{i\in {[n]}} \phi(G_i)$.
		Since $H$ is connected, necessarily, there exists $i\in{[n]}$ such that $H=\phi(G_i)$ and $\phi(G_j) = (\emptyset,\emptyset)$ for every $1\leq j\not = i \leq n$.
	\end{enumerate}
\end{proof}
~\\[-1em]

We now present some preliminary examples and observations to give the reader a better idea concerning the two questions and the difficulties they present. 

\bigskip
First of all, not all graphs are realizable or equivalently, not all trace monoids are hike monoids. The smallest example of a graph that is unrealizable is the square
\begin{equation}
\label{E:Hsquare}
\begin{tikzpicture}[Centering]
\def\radi{0.6}
\node[WhiteLabelNode](a) at (135:\radi){\small$a$};
\node[WhiteLabelNode](b) at (225:\radi){\small$d$};
\node[WhiteLabelNode](c) at (315:\radi){\small$c$};
\node[WhiteLabelNode](d) at (45:\radi){\small$b$};
\draw[line width=0.7](a) to (b);
\draw[line width=0.7](b) to (c);
\draw[line width=0.7](c) to (d);
\draw[line width=0.7](d) to (a);
\node[NullNode](H) at (-1,0){\small$H =$};
\end{tikzpicture}
\end{equation}

\noindent Let us give an intuitive proof of this statement while we defer a more general formal proof of it to Proposition~\ref{observations_prop_cycles}. By trying to build a digraph $G$ such that $\phi(G) = H$, we will necessarily construct additional simple cycles. We begin with a simple cycle of arbitrary length corresponding to vertex~$a$ (\textbf{1.} of Figure \ref{F:SquareConstruction}).
We then add the simple cycles corresponding to $b$ and $d$. They must both share at least one vertex with simple cycle $a$ but not with one another (\textbf{2.} of Figure \ref{F:SquareConstruction}).
We finish by adding simple cycle $c$, which must share a vertex with $b$ and $d$ but not with $a$ (\textbf{3.} of Figure \ref{F:SquareConstruction}).
 By doing so we created two new simple cycles, the internal black one and the external gray one (\textbf{3.} of Figure \ref{F:SquareConstruction}).

\begin{figure}[h]
\vspace{-1em}
\begin{center}

 \begin{tikzpicture}
\def\radi{0.4}
\def\alpha{16.5}
\node[NullNode](oa1) at (-\radi,\radi){\small$a$};
\node[NullNode](oa2) at (\radi,\radi){\small$b$};
\node[NullNode](oa3) at (\radi,-\radi){\small$c$};
\node[NullNode](oa4) at (-\radi,-\radi){\small$d$};
\draw[->] (-\radi,\radi)+(\radi,0) arc (0:360:\radi);
\begin{scope}[shift={(4,0)}]
\node[WhiteNode](b1) at (90:\radi){};
\node[WhiteNode](b2) at (180:\radi){};
\node[NullNode](ob1) at (-\radi,\radi){\small$a$};
\node[NullNode](ob2) at (\radi,\radi){\small$b$};
\node[NullNode](ob3) at (\radi,-\radi){\small$c$};
\node[NullNode](ob4) at (-\radi,-\radi){\small$d$};
\draw[->,gray,line width=0.65] (-\radi,\radi)+(\alpha:\radi)  arc (\alpha:270-\alpha:\radi);
\draw[->,line width=0.65] (-\radi,\radi)+(270+\alpha:\radi) arc (-90+\alpha:-\alpha:\radi);
\draw[->] (\radi,\radi)+(180+\alpha:\radi) arc (180+\alpha:540-\alpha:\radi);
\draw[->](-\radi,-\radi)+(90+\alpha:\radi) arc (90+\alpha:450-\alpha:\radi);
\end{scope}
\begin{scope}[shift={(8,0)}]
\node[WhiteNode](c1) at (90:\radi){};
\node[WhiteNode](c2) at (180:\radi){};
\node[WhiteNode](c3) at (270:\radi){};
\node[WhiteNode](c4) at (0:\radi){};
\node[NullNode](oc1) at (-\radi,\radi){\small$a$};
\node[NullNode](oc2) at (\radi,\radi){\small$b$};
\node[NullNode](oc3) at (\radi,-\radi){\small$c$};
\node[NullNode](oc4) at (-\radi,-\radi){\small$d$};
\draw[->,gray,line width=0.65](-\radi,\radi)+(\alpha:\radi)  arc (\alpha:270-\alpha:\radi);
\draw[->,line width=0.65] (-\radi,\radi)+(270+\alpha:\radi) arc (-90+\alpha:-\alpha:\radi);
\draw[->,gray,line width=0.65](\radi,\radi)+(270+\alpha:\radi) arc (270+\alpha:540-\alpha:\radi);
\draw[->,line width=0.65](\radi,\radi)+(180+\alpha:\radi) arc (180+\alpha:270-\alpha:\radi);
\draw[->,gray,line width=0.65](\radi,-\radi)+(180+\alpha:\radi) arc (180+\alpha:450-\alpha:\radi);
\draw[->,line width=0.65](\radi,-\radi)+(90+\alpha:\radi) arc (90+\alpha:180-\alpha:\radi);
\draw[->,gray,line width=0.65](-\radi,-\radi)+(90+\alpha:\radi) arc (90+\alpha:360-\alpha:\radi);
\draw[->](-\radi,-\radi)+(\alpha:\radi) arc (\alpha:90-\alpha:\radi);
\end{scope}
\node[NullNode](l1) at (-1.2,0){\small\textbf{1.}};
\node[NullNode](l1) at (2.8,0){\small\textbf{2.}};
\node[NullNode](l1) at (6.8,0){\small\textbf{3.}};
\end{tikzpicture}
\end{center}
\vspace{-1.5em}
\caption{The square graph of \eqref{E:Hsquare} is unrealizable: \textbf{1.} First consider simple cycle $a$. \textbf{2.} Cycles~$b$ and $d$ must intersect $a$. \textbf{3.} Concluding with the construction of $d$, we obtain $6$ simple cycles: $a$, $b$, $c$ and $d$ plus the internal black simple cycle and the external gray one.}
\label{F:SquareConstruction}
\end{figure}

 A natural approach could be to propose that a graph comprising a unrealizable graph as induced subgraph may itself be unrealizable. This cannot be so: consider for example the following graph $H$ comprising the unrealizable square (highlighted with gray edges) as induced subgraph. It turns out that $H$ is realized by the bidirected  graph shown on the right,
\begin{equation}\label{EqSquare}
\begin{tikzpicture}[line width=0.8,Centering]
\def\radi{1}
\def\radii{0.5*\radi}
\def\alpha{45}
\def\height{1}
\node[WhiteNode](a) at ({\radi*cos(\alpha)},{\radii*sin(\alpha)}){};
\node[WhiteNode](b) at ({\radi*cos(\alpha+90)},{\radii*sin(\alpha+90)}){};circle(0.1);
\node[WhiteNode](c) at ({\radi*cos(\alpha+180)},{\radii*sin(\alpha+180)}){};circle(0.1);
\node[WhiteNode](d) at ({\radi*cos(\alpha+270)},{\radii*sin(\alpha+270)}){};circle(0.1);
\node[WhiteNode](e) at (0,\height){};
\node[WhiteNode](f) at (0,-\height){};
\draw[color=lightgray,line width=2](a) to (b);
\draw[color=lightgray,line width=2](b) to (c);
\draw[color=lightgray,line width=2](c) to (d);
\draw[color=lightgray,line width=2](d) to (a);
\draw(a) to (e);
\draw(b) to (e);
\draw(c) to (e);
\draw(d) to (e);
\draw(a) to (f);
\draw(b) to (f);
\draw(c) to (f);
\draw(d) to (f);
\draw(e) to (f);
\node[NullNode](l) at (-1.3,0){$H=$};
\begin{scope}[shift={(5,0)}]
\def\radii{0.6*\radi}
\node[WhiteNode](a) at (-\radii,-\radii){}; 
\node[WhiteNode](b) at (\radii,-\radii){}; 
\node[WhiteNode](c) at (\radii,\radii){}; 
\node[WhiteNode](d) at (-\radii,\radii){}; 
\node[NullNode](l) at (1.7,0){$\in\phi^{-1}(H)$};
\path[->](a) edge[bend right=20] (b);
\path[->](b) edge[bend right=20] (c);
\path[->](c) edge[bend right=20] (d);
\path[->](d) edge[bend right=20] (a);
\path[->](a) edge[bend right=20] (d);
\path[->](b) edge[bend right=20] (a);
\path[->](c) edge[bend right=20] (b);
\path[->](d) edge[bend right=20] (c);
\end{scope}
\end{tikzpicture}
\end{equation}
We will generalize this observation later, in Remark~\ref{remrealizsub}.\\[-.5em] 

This suggests that the problem with realizability lies in part with {cycles made of simple cycles}, which must themselves be represented as vertices in $H$. We formalize this intuition and elaborate on it with the following proposition.

\begin{proposition}\label{observations_prop_cycles}
	Let $H=(V_H,E_H)$ be a realizable graph. Then for any induced simple cycle $C=(c_1,\dots, c_n)$ of length at least 4 in $H$, there exist two vertices
	$w_1,w_2\in \modifT{H}{V_H}$ such that: 
	\begin{itemize}
	    \item for all vertices $c_i$ in $C$, edges $\set{w_1,c_i}$ and $\set{w_2,c_i}$ are both in $E_H$.
	    \item the neighbourhoods of $w_1$ and $w_2$ are included in the neighbourhood of $C$. That is to say if $\set{v,w_1}\in E_H$ or $\set{v,w_2}\in E_H$
	    then either $v$ is in $C$ or there exists $c_i$ in $C$ such that $\set{v,c_i}\in \modifT{E}{E_H}$.
	\end{itemize}
	Furthermore, if edge $\set{w_1,w_2}$ is not in $H$ then for any $i\in {[n]}$ there exists a vertex $c_i'$ of $H$ such that edges $\set{c_i,c_i'}$ and
	$\set{c_{i+1}, c_i'}$ are in $E_H$, under the convention 
	$c_{n+1}=c_1$.
\end{proposition}

\begin{proof}
    Let $C=(c_1,\dots, c_n)$ be an induced simple cycle of $H$ of length $n\geq 4$ and let $G=(V_G,E_G)$ be a digraph realizing $H$. Let $v_1, v_2\in V_G$ be two vertices 
    visited by $c_1$ such that $v_1$ is also visited by $c_n$, $v_2$ is also visited by $c_2$ and every other vertex in the path from $v_1$ to $v_2$ along $c_1$ is neither visited by $c_n$ nor by $c_2$. For $i=3$ up to $n$, we recursively define $v_i$ as the first vertex in $c_i$ we encounter when walking along cycle $c_{i-1}$ from $v_{i-1}$. For $i\in{[n]}$ and $v,v'\in c_i$, 
    let us denote by $v\rightarrow_{i} v'$ the path from $v$ to $v'$ along $c_i$ (if $v=v'$ then the path has no edge) and let us show that the rooted closed walk 
    $w_1=v_1\rightarrow_1 v_2\rightarrow_2 v_3 \dots v_n \rightarrow_n v_1$ is a rooted simple cycle. 
    
    First, $w_1$ comprises at least one edge since 
    $n\geq 4$ and $C$ being an induced cycle entails $v_1\not = v_3$. 
    Second, suppose that $w_1$ is not simple and let $v$ be a vertex visited at least twice by $w$. Let $i\in{[n]}$ (under the convention $v_{0}=v_{n}$ and $v_1=v_{n+1}$) be such that $v$ occurs in $v_i\rightarrow_i v_{i+1}$, meaning in particular that $v$ is visited by $c_i$.
    Since $C$ is an 
    induced cycle of $H$, any other appearance of $v$ must be in $v_{j}\rightarrow_{j} v_{j+1}$ with $j=i-1$ or $j=i+1$, as otherwise $\{c_j,c_i\}$ would be a chord of $C$. Assuming $j=i-1$ without loss of
    generality, $v$ is a vertex visited by $c_i$ which also occurs in $v_{i-1}\rightarrow_{i-1} v_i$. But if $i\not = 1$, by construction, $v_i$ is the only vertex of $v_{i-1}\rightarrow_{i-1} v_i$
    which is also visited by $c_i$, implying $v=v_i$ is the only occurance of $v$, contradicting our supposition that $w$ is not simple. If $i=1$, then again by construction $v_1$ is the only vertex of
    $v_1\rightarrow_1 v_2$ which is also visited by $c_n$, hence $v=v_1$ and we have the same conclusion. Consequently, $w_1$ describes 
    a simple cycle satisfying the two claims of the Proposition. 
    
    Let $v_1', v_n'\in V_G$ be two vertices visited by $c_n$ such that $v_1'$ is also visited by $c_1$, $v_n'$ is also visited by $c_{n-1}$ and every other vertex in the path from $v_1$ 
    to $v_n$ along $c_n$ is neither visited by $c_1$ nor by $c_{n-1}$. For $i=n-1$ down to $2$, we recursively define $v_i'$ as the first vertex visited by $c_{i-1}$ we encounter when walking along cycle $c_i$ from $v_{i+1}'$. By the same reasoning as for $w_1$, walk $w_2=v_1'\rightarrow_n v_n'\rightarrow_{n-1} v_{n-1}' \dots v_2' \rightarrow_1 v_1'$ is a simple cycle satisfying the 
    two claims of the Proposition. Furthermore, $w_2$ crosses cycles $c_1,\dots c_n$ in reverse order as compared to $w_1$ and so is distinct from it.
    
    We now turn to the situation where edge $\set{w_1,w_2}$ is not in $E_H$, i.e., $w_1$ and $w_2$ have no vertex in common in $G$. Let $i\in{[n]}$. By construction, vertices $v_i$ and $v_i'$
    are both in $c_i$ and $c_{i-1}$ and are distinct since $v_i$ is in $w_1$ and $v_i'$ is in $w_2$. Let us consider cycle 
    $\mathfrak{c}:=v_i\rightarrow_{i-1}v_i'\rightarrow_i v_i$ and show that $v_{i-1}$ can not be visited by it. First, $v_{i-1}$ can not be visited by path $v_i'\rightarrow_i v_i$ since it is in $c_{i-2}$, which
    is disjoint from $c_i$. It can also not be in path $v_i\rightarrow_{i-1}v_i'$ as otherwise $v_i'$ would be in $v_{i-1}\rightarrow_{i-1} v_i$, which is a part of $w_1$. Then $v_{i-1}$ is not visited by $\mathfrak{c}$. By the
    same reasoning we find that $v_{i+1}'$ is not visited by $\mathfrak{c}$ either. As a consequence $\mathfrak{c}$ is distinct from $c_{i-1}$ and $c_i$ and shares at least one vertex with each of them. If $\mathfrak{c}$ is simple
    then we choose $c_i':=\mathfrak{c}$ equal to this cycle, else we choose $c_i'$ to be a simple cycle composing $\mathfrak{c}$ and visiting $v_i$.
\end{proof}
~

Triangles themselves are no obstacle to realizability as any triangle is realizable by a bouquet of three self-loops on the same vertex. This does not entail that any chordal graph is realizable however. For example, the triforce graph: 
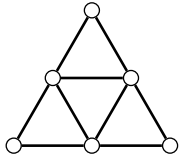
\vspace{-2mm}
\begin{equation*}
\begin{tikzpicture}[line width=1,Centering]
\def\radi{0.6}
\node[WhiteNode](a) at (30:\radi){};
\node[WhiteNode](b) at (150:\radi){};
\node[WhiteNode](c) at (270:\radi){};
\node[WhiteNode](d) at (90:2*\radi){};
\node[WhiteNode](e) at (210:2*\radi){};
\node[WhiteNode](f) at (330:2*\radi){};
\draw(a) to (b);
\draw(b) to (c);
\draw(c) to (a);
\draw(a) to (d);
\draw(d) to (b);
\draw(b) to (e);
\draw(e) to (c);
\draw(c) to (f);
\draw(f) to (a);
\end{tikzpicture}
\end{equation*}
 \noindent is unrealizable for the same reasons barring the square from being realized. In fact, this graph can be seen as the particular case of Proposition~\ref{observations_prop_cycles}
 for cycles of length $3$. Indeed, the reader may attempt to build a digraph with a triforce dependency graph to see that because of the $3$ external vertices, the cycles composing the
 central triangle necessarily imply the existence of two further cycles sharing vertices with every other cycle hence a pair of vertices connected to all vertices of the triforce, contradicting its structure.
 
All observations made so far point to realizability as depending solely on the induced cycles of $H$ and their neighborhood, with the need to distinguish long (length $\ell \geq 4$) from short ones (exception of the triforce graph). This is in fact 
not true; for example we prove later in {\ref{AppA}} that the following graph is unrealizable yet possesses neither an induced cycle of length $\ell \geq 4$ nor a triforce of triangles:
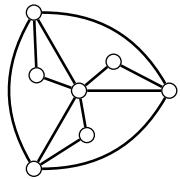
\begin{equation*}
\begin{tikzpicture}[Centering,line width=1]
\def\radi{1.2}
\def\radii{0.6}
\def\alpha{40}
\node[WhiteNode](c) at (0,0){};
\node[WhiteNode](b1) at (\alpha:\radii){};
\node[WhiteNode](b2) at (120+\alpha:\radii){};
\node[WhiteNode](b3) at (240+\alpha:\radii){};
\node[WhiteNode](a1) at (0:\radi){};
\node[WhiteNode](a2) at (120:\radi){};
\node[WhiteNode](a3) at (240:\radi){};
\draw[-](c) to (a1);
\draw[-](c) to (a2);
\draw[-](c) to (a3);
\draw[-](c) to (b1);
\draw[-](c) to (b2);
\draw[-](c) to (b3);
\draw[-](a1) to (b1);
\draw[-](a2) to (b2);
\draw[-](a3) to (b3);
\path[-](a1) edge[bend right=30] (a2);
\path[-](a2) edge[bend right=30] (a3);
\path[-](a3) edge[bend right=30] (a1);
\end{tikzpicture}
\end{equation*}

\section{Surjectivity: which graphs are realizable?}\label{Surjec}

As shown by the preliminary examples and results of Section~\ref{section_observations}, understanding the image of $\phi$ is more complicated than it may first appear. Ideally, one would like to have a criterion for deciding if a graph is realizable, i.e., if it is in the image of $\phi$.  Such a criterion remains elusive and if it exists, it must be highly non-trivial as we will demonstrate. Nonetheless we here establish that realizability is \emph{decidable} by providing an 
algebraic condition that is equivalent to it. While this algebraization does not in itself shed additional light on realizability  it leads to an algorithm for systematically checking for it.\\[-.5em] 

Given a digraph $G=(V,E)$ and $H=\phi(G)$ its hike dependency graph we have, by definition, that the vertices of $H$ correspond to the simple cycles of $G$. Conversely we may ask under what form do the vertices of $G$ manifest themselves in $H$? 
Given $v$ a vertex of $G$, all simple cycles of $G$ visiting $v$ share at least this vertex hence do not commute in $\mathcal{H}_G$. This implies that they are all mutually connected by edges in $H$, i.e.\modifJ{}{,} they form a clique $\kappa_v$. The fact that vertices of $G$ {yield} cliques in $H$ leads to the following observations:
\begin{enumerate}[label=\textit{\roman*})]
	\item\label{surjectivity_obs1} The set $\set{\kappa_v}_{v\in V}$ is a {\em clique cover} of $H$, that is each clique $\kappa_v$ is a subgraph of $H$ and every 
	edge and vertex of $H$ appears in at last one $\kappa_v$.
	\item\label{surjectivity_obs2} For $W\subseteq V$, the set
	$\bigcap_{v\in W} \kappa_v \setminus \bigcup_{v\in V\setminus W} \kappa_v$ 
	corresponds to the simple cycles of $G$ with exactly $W$ as vertex set.
\end{enumerate}

Let us now give a criterion equivalent to realizability. For $S$ a set, we denote by $\Cyc_S$ the set of permutations over $S$ with cycle decomposition of
length $1$. 

\begin{theorem}\label{surjectivity_eq_condition}
	Let $H$ be a graph. Then $H$ is realizable if and only if there exists a clique cover $\set{\kappa_1,\dots,\kappa_n}$ of $H$ such that the following polynomial 
	system in variables $(m_{ij})_{i,j\in{[n]}}$ admits an integer solution:
	\begin{equation}\label{surjectivity_pol_system}
		\forall\, W\subseteq{[n]},\quad \sum_{\sigma\in\Cyc_W} \prod_{v\in W} m_{v,\sigma(v)} = \big|\mathcal{K}_W\big|,
	\end{equation}
	where $\mathcal{K}_W:=\bigcap_{v\in W} \kappa_v \setminus \bigcup_{v\in {[n]}\setminus W} \kappa_v$. 
	
	In this case, $H$ is realized by the digraph $G$ with vertex set ${[n]}$ and adjacency matrix $\mat{A}$ defined by $\mat{A}_{i,j} := m_{i,j}$ for
	$1\leq i,j\leq n$.
\end{theorem}

{\begin{corollary}\label{realizability_decidable}
    Realizability is decidable.
\end{corollary}}

\begin{proof}[\modifJ{}{Proof of Theorem~\ref{surjectivity_eq_condition}}]

	Let us first suppose that $H$ is realizable and let $G$ be a digraph realizing $H$ with vertex set ${[n]}$ and adjacency matrix $\mat{A}$. From 
	observation \ref{surjectivity_obs1} above, set $\set{\kappa_v}_{v\in{[n]}}$ is a clique cover of $H$. Furthermore, from  \ref{surjectivity_obs2} it follows that for any given subset $W\subseteq {[n]}$, the number of cycles with vertex set $W$ is equal to $\card{\mathcal{K}_W}$. Now 
	remark that this number is also equal to $\sum_{\sigma\in\Cyc_W} \prod_{v\in W} \mat{A}_{v,\sigma(v)}$. Indeed, to choose a simple cycle with vertex 
	set $W$ in $G$, we must first choose in which cyclic order we visit the vertices in $W$, that is a permutation in $\Cyc_W$, and then, for every vertex 
	$v\in W$ choose which edge to follow from $v$ to $\sigma(v)$ among the $\mat{A}_{v,\sigma(v)}$ possibilities. Hence system (\ref{surjectivity_pol_system}) 
	admits $m_{i,j} = \mat{A}_{i,j}$ as an integer solution.

	Let us now suppose that system (\ref{surjectivity_pol_system}) admits an integer solution and let $G$ be the digraph with vertex set ${[n]}$ and adjacency 
	matrix $\mat{A}$ such that $\mat{A}_{i,j} = m_{i,j}$ for $1\leq i,j\leq n$. Let $H':=\phi(G)$ be the hike dependency graph of $G$ and for every $v\in {[n]}$ let $\kappa_v'$ 
	be the clique in $H'$ corresponding to the simple cycles of $G$ visiting $v$.
	As for $\mathcal{K}_W$, we define $\mathcal{K'}_W$ to be $\bigcap_{v\in W} \kappa_v \setminus \bigcup_{v\in {[n]}\setminus W} \kappa_v$.
	Then from what precedes, we know that for any $W\subseteq {[n]}$, the number 
	of vertices in $\mathcal{K'}_W$ is equal to $\sum_{\sigma\in\Cyc_W} \prod_{v\in W} \mat{A}_{v,\sigma(v)}$. But since 
	$(\mat{A}_{i,j})_{i,j\in{[n]}} = (m_{i,j})_{i,j\in{[n]}}$ is a solution of (\ref{surjectivity_pol_system}), this is also equal to the number of vertices in 
	$\mathcal{K}_W$. Consequently, there is a bijection $\psi_W$ between these two sets $\mathcal{K}_W$ and $\mathcal{K'}_W$ for every $W\subseteq {[n]}$. 
	In addition, because 
	$\left(\mathcal{K'}_W\right)_{W\subseteq{[n]}}$ and $\left(\mathcal{K}_W\right)_{W\subseteq{[n]}}$ are partitions of the vertex sets of $H'$ and of $H$, respectively, there exists a global bijection $\psi$ between the vertices of $H'$ and of $H$. 
	In the same way, since for any $u\in {[n]}$ the sequences
	$\left(\mathcal{K'}_W\right)_{u\in W\subseteq{[n]}}$ and $\left(\mathcal{K}_W\right)_{u\in W\subseteq{[n]}}$ are  partitions
	of $\kappa_u'$ and of $\kappa_u$, respectively, $\psi$ realizes a bijection between $\kappa_u'$ and $\kappa_u$. This makes $\psi$ into an isomorphism of graphs. Indeed, let be $v_1$ and $v_2$ be two 
	vertices of~$H$, then we have
	\begin{equation*}\begin{split}
		\text{edge $\set{v_1,v_2}$ is in $H'$} &\iff \exists\, v\in{[n]}, v_1,v_2\in \kappa_v' \\
		&\iff  \exists\, v\in{[n]}, \psi(v_1),\psi(v_2)\in \kappa_v \\
		&\iff \text{edge $\set{\psi(v_1),\psi(v_2)}$ is in $H$}.
	\end{split}\end{equation*}
	Whence $H$ is realizable and $G$ realizes $H$.
\end{proof}

Before giving examples of {Theorem~\ref{surjectivity_eq_condition}}, let us mention that we can ignore clique covers such that $\kappa_i\subseteq \kappa_j$ for some $i\neq j$. Indeed, if for such a clique cover
the system (\ref{surjectivity_pol_system}) admits an integer solution, then no cycles would pass by the vertex $i$ in the associated digraph, since we have 
$|\mathcal{K}_W| = 0$ for any $W$ containing $i$. Hence the digraph obtained by removing vertex $i$ still realizes $H$ and corresponds to the clique cover obtained
by removing the clique $\kappa_i$. In the sequel, we call trivial any clique $\kappa_i$ such that $\kappa_i\subseteq \kappa_j$ for some $j\neq i$.\\

Let us now illustrate the theorem with three concrete cases. 
\begin{examples}\label{surjectivity_ex1}
~
    \begin{itemize}
	\item Let $H$ be the square over $\set{a,b,c,d}$ as in  (\ref{E:Hsquare}).

	The only clique cover of $H$ with no trivial clique is $\kappa_1=\set{a,b}$, $\kappa_2=\set{b,c}$, $\kappa_3=\set{c,d}$ and $\kappa_4=\set{d,a}$. 
    System \eqref{surjectivity_pol_system} comprises an equation for each subset $W$ of $\{1,2,3,4\}$.
    Firstly, considering $W=\emptyset$ we get the trivial equation $0=0$. Secondly, for $W=\set{1}$, we find
    \[
    m_{1,1}=\left|\kappa_1\setminus (\kappa_2\cup\kappa_3\cup\kappa_4)\right|=\left|\set{a,b}\setminus\set{a,b,c,d}\right|=0,
    \]
    and repeating the same reasoning for $W=\set{2},\set{3}$ and $\set{4}$ we obtain
    \begin{equation*}
    m_{1,1}=m_{2,2}=m_{3,3}=m_{4,4}=0.
    \end{equation*}
    Thirdly, for $W=\set{1,2}$, we find $m_{1,2} m_{2,1}=\left|(\kappa_1\cap\kappa_2)\setminus(\kappa_3\cup\kappa_4)\right|=\left| \set{b}\setminus\set{d}\right|=1$.
Similarly $W=\set{2,3},\set{3,4}$ and $\set{1,4}$ yield
\begin{equation}
\label{E:Thm:Ex1-1}
    m_{1,2}m_{2,1}=m_{2,3}m_{3,2}=m_{3,4}m_{4,3}=m_{1,4}m_{4,1}=1.
\end{equation}
As $\kappa_1\cap\kappa_3=\kappa_2\cap\kappa_4=\emptyset$, the set $\mathcal{K}_W$ is empty whenever $\set{1,3}$ or $\set{2,4}$ is a subset of $W$: 
\begin{align*}
 m_{1,3}m_{3,1}=m_{2,4}m_{4,2}&=0 &&\text{for $W=\set{1,3},\set{2,4}$},\\
 m_{1,2}m_{2,3}m_{3,1}+m_{1,3}m_{3,2}m_{2,1}&=0 &&\text{for $W=\set{1,2,3}$},\\
 m_{1,2}m_{2,4}m_{4,1}+m_{1,4}m_{4,2}m_{2,1}&=0 &&\text{for $W=\set{1,2,4}$},\\
 m_{1,3}m_{3,4}m_{4,1}+m_{1,4}m_{4,3}m_{3,1}&=0 &&\text{for $W=\set{1,3,4}$},\\
 m_{2,3}m_{3,4}m_{4,2}+m_{2,4}m_{4,3}m_{3,2}&=0 &&\text{for $W=\set{2,3,4}$}.
 \end{align*}
 Lastly, for $W=\{1,2,3,4\}$, we get
\begin{equation}
\label{E:Thm:Ex1-2}
m_{1,2}m_{2,3}m_{3,4}m_{4,1}+m_{1,2}m_{2,4}m_{4,3}m_{3,1}+m_{1,3}m_{3,2}m_{2,4}m_{4,1}+m_{1,3}m_{3,4}m_{4,2}m_{2,1}=0.
\end{equation}

	This system does not have any integer solution since \eqref{E:Thm:Ex1-1} forces $m_{1,2} = m_{2,3} = m_{3,4} = m_{4,1} = 1$.
	In particular $m_{1,2}m_{2,3}m_{3,4}m_{4,1}$ equals $1$, which is impossible by \eqref{E:Thm:Ex1-2}. We recover the observation that the square is unrealizable.
	\item Let $H$ be the following graph over $\set{a,b,c,d,e}$ with clique cover $\set{\kappa_1,\kappa_2,\kappa_3}$:
\begin{equation*}
\begin{tikzpicture}[Centering,line width=1]
\def\radi{1}
\node[WhiteLabelNode](e) at (0:\radi){\small$e$};
\node[WhiteLabelNode](d) at (-0.2,0){\small$d$};
\node[WhiteLabelNode](a) at (90:\radi){\small$a$};
\node[WhiteLabelNode](c) at (180:\radi){\small$c$};
\node[WhiteLabelNode](b) at (270:\radi){\small$b$};
\draw(a) to (c);
\draw(c) to (b);
\draw(b) to (d);
\draw(d) to (a);
\draw(b) to (e);
\draw(e) to (a);
\path(a) edge[bend left=20] (b);
\node[NullNode](H) at (-1.7,0){\small$H=$};
\begin{scope}[shift={(4.5,0)}]
\node[WhiteLabelNode](a) at (90:\radi){\small$a$};
\node[WhiteLabelNode](c) at (180:\radi){\small$c$};
\node[WhiteLabelNode](b) at (270:\radi){\small$b$};
\draw(a) to (c);
\draw(c) to (b);
\path(a) edge[bend left=20] (b);
\node[NullNode](H) at (-1.7,0){\small$\kappa_1=$};
\end{scope}
\begin{scope}[shift={(7,0)}]
\node[WhiteLabelNode](d) at (-0.2,0){\small$d$};
\node[WhiteLabelNode](a) at (90:\radi){\small$a$};
\node[WhiteLabelNode](b) at (270:\radi){\small$b$};
\draw(b) to (d);
\draw(d) to (a);
\path(a) edge[bend left=20] (b);
\node[NullNode](H) at (-1,0){\small$\kappa_2=$};
\end{scope}
\begin{scope}[shift={(9,0)}]
\node[WhiteLabelNode](e) at (0:\radi){\small$e$};
\node[WhiteLabelNode](a) at (90:\radi){\small$a$};
\node[WhiteLabelNode](b) at (270:\radi){\small$b$};
\draw(b) to (e);
\draw(e) to (a);
\path(a) edge[bend left=20] (b);
\node[NullNode](H) at (-0.5,0){\small$\kappa_3=$};
\end{scope}
\end{tikzpicture}
\end{equation*}
Then, considering $\emptyset\not=W\subseteq \{1,2,3\}$ we have   
\[
\begin{array}{c|c}
W & \mathcal{K}_W\\
\hline
\set{1} & \kappa_1\setminus(\kappa_2\cup\kappa_3)=\set{c}\\
\set{2} & \kappa_2\setminus(\kappa_1\cup\kappa_3)=\set{d}\\
\set{3} & \kappa_3\setminus(\kappa_1\cup\kappa_2)=\set{e}\\
\set{1,2} & (\kappa_1\cap\kappa_2)\setminus\kappa_3=\emptyset\\
\set{1,3} & (\kappa_1\cap\kappa_3)\setminus\kappa_2=\emptyset \\
\set{2,3} & (\kappa_2\cap\kappa_3)\setminus\kappa_1=\emptyset \\
\set{1,2,3} & \kappa_1\cap\kappa_2\cap\kappa_3=\set{a,b}
\end{array}
\]
and system (\ref{surjectivity_pol_system}) is given by:
	\begin{enumerate}[label=\roman*)]
		\item $m_{1,1} = m_{2,2} = m_{3,3} = 1$,
		\item $m_{1,2}m_{2,1} = m_{2,3}m_{3,2} = m_{3,1}m_{1,3} = 0$,
		\item $m_{1,2}m_{2,3}m_{3,1} + m_{1,3}m_{3,2}m_{2,1} = 2$.
	\end{enumerate}
	This system admits the following solution: $m_{1,1} = m_{2,2} = m_{3,3} = m_{1,2} = m_{2,3} = 1$, $m_{3,1} = 2$ and $m_{1,3} = m_{3,2} = m_{2,1} = 0$, as 
	well as 5 further solutions (obtained on changing which of the monomials $m_{1,2}m_{2,3}m_{3,1}$ and $m_{1,3}m_{3,2}m_{2,1}$ is null and which variable $m_{i,j}$ is equal to 2). This indicates that $H$ is realizable by digraph $G$ with adjacency matrix $\mat{A}_{i,j} = m_{i,j}$, that is
	
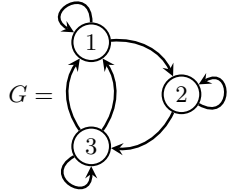
\begin{equation*}
    \begin{tikzpicture}[Centering,line width=1]
    \def\radi{0.8}
    \node[WhiteLabelNode](1) at (120:\radi){\small$1$};
    \node[WhiteLabelNode](2) at (0:\radi){\small$2$};
    \node[WhiteLabelNode](3) at (240:\radi){\small$3$};
    \path[->](1) edge[bend left=35] (2);
    \path[->](2) edge[bend left=35] (3);
    \path[->](3) edge[bend left=35] (1);
    \path[->](3) edge[bend right=35] (1);
    \path[->](1) edge [out=90,in=150,distance=15] (1);
    \path[->](2) edge [out=-30,in=30,distance=15] (2);
    \path[->](3) edge [out=210,in=270,distance=15] (3);
    \node[NullNode](l) at (-1.2,0){\small$G=$};
    \end{tikzpicture}
	\end{equation*}
	for which we verify that $\phi(G)=H$ holds as predicted.
	\item 
Let $H$ be the following graph over $\set{a,b,c,d,e,f,g}$ with clique cover $\set{\kappa_1,\kappa_2,\kappa_3}$:
    
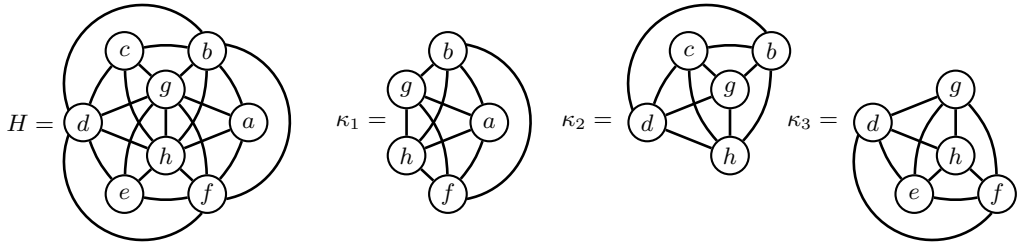
\begin{equation*}
        \begin{tikzpicture}[Centering,line width=1]
        \def\radi{1.1}
        \node[WhiteLabelNode] (a) at (0:\radi){\small$a$};
        \node[WhiteLabelNode] (b) at (60:\radi){\small$b$};
        \node[WhiteLabelNode] (c) at (120:\radi){\small$c$};
        \node[WhiteLabelNode] (d) at (180:\radi){\small$d$};
        \node[WhiteLabelNode] (e) at (240:\radi){\small$e$};
        \node[WhiteLabelNode] (f) at (300:\radi){\small$f$};
        \node[WhiteLabelNode] (g) at (90:0.4*\radi){\small$g$};
        \node[WhiteLabelNode] (h) at (-90:0.4*\radi){\small$h$};
        \path(a) edge[bend right=10] (b);
        \path(b) edge[bend right=10] (c);
        \path(c) edge[bend right=10] (d);
        \path(d) edge[bend right=10] (e);
        \path(e) edge[bend right=10] (f);
        \path(f) edge[bend right=10] (a);
        \draw (d)+(220:0.25) arc (160:319:0.95*\radi);
        \draw (d)+(140:0.25) arc (200:41:0.95*\radi);
        \draw (b)+(20:0.25) arc (80:41-120:0.95*\radi);
        \draw (g) to (h);
        \draw (d) to (g);
        \draw (d) to (h);
        \draw (a) to (g);
        \draw (a) to (h);
        \draw (b) to (g);
        \draw (c) to (g);
        \draw (e) to (h);
        \draw (f) to (h);
        \path(b) edge[bend left=20] (h);
        \path(c) edge[bend right=20] (h);
        \path(e) edge[bend left=20] (g);
        \path(f) edge[bend right=20] (g);
        \node[NullNode] (l) at (-1.8,0){\small $H=$};
        \begin{scope}[shift={(3.2,0)}]
        \node[WhiteLabelNode] (a) at (0:\radi){\small$a$};
        \node[WhiteLabelNode] (b) at (60:\radi){\small$b$};
        \node[WhiteLabelNode] (f) at (300:\radi){\small$f$};
        \node[WhiteLabelNode] (g) at (90:0.4*\radi){\small$g$};
        \node[WhiteLabelNode] (h) at (-90:0.4*\radi){\small$h$};
        \path(a) edge[bend right=10] (b);
        \path(f) edge[bend right=10] (a);
        \draw (b)+(20:0.25) arc (80:41-120:0.95*\radi);
        \draw (g) to (h);
        \draw (a) to (g);
        \draw (a) to (h);
        \draw (b) to (g);
        \draw (f) to (h);
        \path(b) edge[bend left=20] (h);
        \path(f) edge[bend right=20] (g);
        \node[NullNode] (l) at (-0.6,0){\small $\kappa_1=$};
        \end{scope}
        \begin{scope}[shift={(7.5,0)}]
        \node[WhiteLabelNode] (b) at (60:\radi){\small$b$};
        \node[WhiteLabelNode] (c) at (120:\radi){\small$c$};
        \node[WhiteLabelNode] (d) at (180:\radi){\small$d$};
        \node[WhiteLabelNode] (g) at (90:0.4*\radi){\small$g$};
        \node[WhiteLabelNode] (h) at (-90:0.4*\radi){\small$h$};
        \path(b) edge[bend right=10] (c);
        \path(c) edge[bend right=10] (d);
        \draw (d)+(140:0.25) arc (200:41:0.95*\radi);
        \draw (g) to (h);
        \draw (d) to (g);
        \draw (d) to (h);
        \draw (b) to (g);
        \draw (c) to (g);
        \path(b) edge[bend left=20] (h);
        \path(c) edge[bend right=20] (h);
        \node[NullNode] (l) at (-1.9,0){\small $\kappa_2=$};
        \end{scope}
        \begin{scope}[shift={(10.5,0)}]
        \node[WhiteLabelNode] (d) at (180:\radi){\small$d$};
        \node[WhiteLabelNode] (e) at (240:\radi){\small$e$};
        \node[WhiteLabelNode] (f) at (300:\radi){\small$f$};
        \node[WhiteLabelNode] (g) at (90:0.4*\radi){\small$g$};
        \node[WhiteLabelNode] (h) at (-90:0.4*\radi){\small$h$};
        \path(d) edge[bend right=10] (e);
        \path(e) edge[bend right=10] (f);
        \draw (d)+(220:0.25) arc (160:319:0.95*\radi);
        \draw (g) to (h);
        \draw (d) to (g);
        \draw (d) to (h);
        \draw (e) to (h);
        \draw (f) to (h);
        \path(e) edge[bend left=20] (g);
        \path(f) edge[bend right=20] (g);
        \node[NullNode] (l) at (-1.9,0){\small $\kappa_3=$};
        \end{scope}
        \end{tikzpicture}
    \end{equation*}
    The system (\ref{surjectivity_pol_system}) is given by:
	\begin{enumerate}[label=\roman*)]
		\item $m_{1,1} = m_{2,2} = m_{3,3} = 1$,
		\item $m_{1,2}m_{2,1} = m_{2,3}m_{3,2} = m_{3,1}m_{1,3} = 1$,
		\item $m_{1,2}m_{2,3}m_{3,1} + m_{1,3}m_{3,2}m_{2,1} = 2$.
	\end{enumerate}
	This system admits the solution: $m_{i,j} = 1$ for every $1\leq i,j\leq 3$ The digraph $G$ with adjacency matrix $\mat{A}_{i,j} = 1$ is then:
	
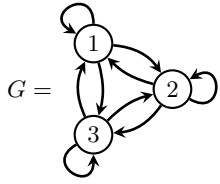
\begin{equation*}
        \begin{tikzpicture}[Centering,line width=1]
        \def\radi{0.7}
        \node[WhiteLabelNode](1) at (120:\radi){\small$1$};
        \node[WhiteLabelNode](2) at (0:\radi){\small$2$};
        \node[WhiteLabelNode](3) at (240:\radi){\small$3$};
        \path[->](1) edge[bend left=25] (2);
        \path[->](2) edge[bend left=25] (3);
        \path[->](3) edge[bend left=25] (1);
        \path[->](1) edge[bend left=15] (3);
        \path[->](3) edge[bend left=15] (2);
        \path[->](2) edge[bend left=15] (1);
        \path[->](1) edge [out=90,in=150,distance=15] (1);
        \path[->](2) edge [out=-30,in=30,distance=15] (2);
        \path[->](3) edge [out=210,in=270,distance=15] (3);
        \node[NullNode](l) at (-1.2,0){\small$G=$};
        \end{tikzpicture}
	\end{equation*}
	and we find that $H$ is the hike dependency graph of $G$ as requested.
	\end{itemize}
\end{examples}

We can deduce the following corollary from Theorem~\ref{surjectivity_eq_condition}.
\begin{corollary}
    All trees are realizable.
\end{corollary}

\begin{proof}
    Let $T=(V,E)$ be a tree. The cliques of a tree are exactly its edges and its vertices, hence the unique clique cover of $T$ with no trivial clique is its set of edges
    $E=\set{e_1,\dots, e_n}$. For each $v\in V$, denote by $E(v)$ the set of edges containing $v$. The polynomial system (\ref{surjectivity_pol_system}) is then
    given by:
    $$\forall\, W\subseteq E,\quad \sum_{\sigma\in\Cyc_W} \prod_{e\in W} m_{e,\sigma(e)} = \left\{\begin{array}{cl}
    1  & \text{if $\exists v\in V,\, W = E(v)$}   \\ 
    0  & \text{else}\end{array}\right..$$
    A solution of this system consist in choosing a permutation $\sigma_{v}\in\Cyc_{E(v)}$ for each $v\in V$ and assigning 1 to each $m_{e,\sigma_v(e)}$ for $e\in E(v)$
    and 0 to all others $m_{i,j}$.
\end{proof}

\begin{remark}
    The digraph obtained in the above proof realizing a tree $T$ is a directed medial graph of $T$.
\end{remark}

\modifT{Beyond trees, Theorem~\ref{surjectivity_eq_condition} indicates that}{As mentionned in Corollary~\ref{realizability_decidable},} realizability is decidable: an algorithm checking for it can work through all possible cliques covers of a graph $H$,
verifying for each such cover if the accompanying system~(\ref{surjectivity_pol_system}) admits an integer solution. In practice, the search for cliques covers can be highly time consuming
since we are searching all the possible ways to cover a set (the edges of $H$) with a set of subsets (the cliques). This algorithm could be greatly accelerated by restricting the pool of
cliques covers. In fact, in Example~\ref{surjectivity_ex1}, we always presented a clique cover with minimal number of cliques. This seems to be always possible hence the following
conjecture.

\begin{conjecture}\label{surjectivity_conj}
	Let $H$ be a realizable graph. Then we can always find a digraph $G$ realizing $H$ with number of vertices $|V_G|$ equal to the smallest amount of cliques
	necessary to cover $H$.
	That is to say, Theorem~\ref{surjectivity_eq_condition} continues to hold if we require that $n$ be minimal. 
\end{conjecture}

The fact that finding a clique cover with minimal number of cliques such that system~(\ref{surjectivity_pol_system}) has an integer solution implies realizability is true. It is the converse assertion that is difficult to prove: does realizability always imply the existence of a {\em minimal} clique cover? The natural approach to prove this would 
be to suppose that a graph $H$ is realizable by a digraph $G$ and then transform $G$ so as to preserve its hike dependency graph while lowering its number of vertices. This is non-trivial and directly related to the question of $\phi$-injectivity which we study in the next section.\\[-.5em]

\begin{remark}
    Contrary to the second and third cases of Examples~\ref{surjectivity_ex1}, even with Conjecture~\ref{surjectivity_conj} true, the clique cover leading to an integer solution of (\ref{surjectivity_pol_system}) is not necessarily made
    of maximal cliques. For $H_1=(V_1,E_1)$ and $H_2=(V_2,E_2)$ two graphs, let us denote by $H_1\!\bowtie \! H_2$ the graph with vertex set $V_1\sqcup V_2$ and edge set 
    $E_1\sqcup E_2\sqcup (V_1\times V_2$). Let $H_1$ be the square graph on vertices $\set{a,b,c,d}$, $H_2$ be the complete graph on 4 vertices $\set{e,f,g,h}$ and $H$ be the graph $H_1\!\bowtie \! H_2$.
     On $H$, the minimal size of a clique cover is $4$, since every clique is a subgraph of a maximal clique and $H$ has 4 maximal cliques: $\set{a,b,e,f,g,h}$, $\set{b,c,e,f,g,h}$, $\set{c,d,e,f,g,h}$ and $\set{d,a,e,f,g,h}$. One can check that system~(\ref{surjectivity_pol_system}) does not have 
    an integer solution if we choose the maximal cliques to form the clique cover, yet it has an integer solution if we choose instead cliques $\set{a,b,e,f,g}$, 
    $\set{b,c,e,f,g,h}$, $\set{c,d,e,f,h}$ and $\set{d,a,e,f,g,h}$ to form the clique cover. This solution leads to the following digraph:
    
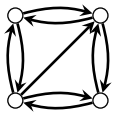
\begin{equation*}
        \begin{tikzpicture}[Centering,line width=1]
        \def\radi{0.8}
        \node[WhiteNode](1) at (135:\radi){};
        \node[WhiteNode](2) at (45:\radi){};
        \node[WhiteNode](3) at (315:\radi){};
        \node[WhiteNode](4) at (225:\radi){};
        \draw[->](4) to (2);
        \path[->](1) edge[bend left=15] (2);
        \path[->](2) edge[bend left=15] (1);
        \path[->](2) edge[bend left=15] (3);
        \path[->](3) edge[bend left=15] (2);
        \path[->](3) edge[bend left=15] (4);
        \path[->](4) edge[bend left=15] (3);
        \path[->](1) edge[bend left=15] (4);
        \path[->](4) edge[bend left=15] (1);
        \end{tikzpicture}
    \end{equation*}
    \vspace{5mm}
\end{remark}

Using Proposition~\ref{observations_prop_cycles} to test for non-realizability, Conjecture~\ref{surjectivity_conj} to test for realizability and with ad-hoc arguments in 9 remaining undecided cases, we computed the number of unlabelled connected realizable graphs with up to $7$ vertices as (OEIS \OEIS{A348365}, computations based on the list of unlabelled connected graphs available at \url{http://users.cecs.anu.edu.au/~bdm/data/graphs.html} and \cite{McKay2014}\modifJ{.}{}) 
\[
1,\,1,\,2,\,5,\,15,\,58,\,265.
\] 
In comparison, the number of unlabelled connected graphs on up to $7$ vertices is known to be (OEIS \OEIS{A001349}) 
\[
1,\,1,\,2,\,6,\,21,\,112,\,853{.}
\]
Whilst most graphs are realizable for small number of vertices this changes when more vertices are considered. 
We conjecture that the proportion of realizable graphs on $n$ vertices with respect to all unlabelled connected graphs on $n$ vertices decreases as $n$ increases. At the same time, we know that:

\begin{proposition}
Let $h_n$ be the number of connected realizable graphs on $n\geq 1$ vertices. Then the sequence $\{h_n\}_{n\geq 1}$ is strictly increasing, and grows at least exponentially with $n$ as $n\to \infty$. 
\end{proposition}
\begin{proof}
Let $H_1$ and $H_2$ be two realizable graphs on $n\geq 1$ and $m\geq 1$ vertices, respectively, and  
let $G_1\in\phi^{-1}(H_1)$ and $G_2\in\phi^{-1}(H_2)$. Let $v_1$ be a vertex of $G_1$ and $v_2$ be a vertex of $G_2$ and consider the glued digraph $G_3$ obtained by taking the disjoint union of
$G_1$ and $G_2$ and fusing $v_1$ with $v_2$. Then $H_3=\phi(G_3)$ is obtained by taking the disjoint union of $H_1$ and $H_2$ and adding an edge between each vertex
of $H_1$ corresponding to a cycle containing $v_1$ and each vertex of $H_2$ corresponding to a cycle containing $v_2$. In particular, since no simple cycle was created in $G_3$ that is not already in $G_1$ or $G_2$, $H_3$ has $n+m$ vertices and 
$h_{n+m} \geq h_n\,h_m$. From this we take the logarithm to get $\log h_{n+m}\geq \log h_n + \log h_m$, i.e., the sequence $\{\log h_n\}_{n\geq 1}$ is superadditive. Then by Fekete's lemma $\log \eta :=\lim_{n\to\infty} n^{-1} \log  h_{n} $ exists and is given by $\sup_{n\geq 1} n^{-1} \log  h_{n}$. If this $\sup$ is finite, $\eta$ is a finite constant and $h_n$ grows exponentially as $n \to \infty$. If the sup is infinite then $h_n$ grows superexponentially as $n\to\infty$. 

Second, setting $m=1$ leads to $h_{n+1}\geq h_n h_1=h_n$ as there is only one graph $H$ on one vertex, which is realized by the graph $G_1$ made of a single simple cycle $c_1$ of arbitrary length. Now consider the cone graph $H_{n+1}$ comprising an induced cycle on $n-1$ vertices with two additional vertices connected to all others. This graph cannot be obtained as the hike dependency graph of a digraph obtained by gluing $G_1$ to another digraph $G'$. Indeed, $c_1$ can neither be one of the two summits of the cone because by Proposition~\ref{observations_prop_cycles}, $H_{n+1}\backslash c_1$ would not be realizable; nor can $c_1$ be a part of the long induced cycle in $H_{n+1}$ as all  neighbors of $c_1$ would form a clique.
Then $h_{n+1}\geq h_n +1>h_n$ and the sequence is strictly increasing.
\end{proof}

The non-realizability of most of the graphs on up to 7 vertices is decided thanks to Proposition~\ref{observations_prop_cycles}, which relies on induced cycles of simple cycles. 
Overall, when evaluating $h_7$, we are left with 9 graphs for which ad-hoc arguments must be found as their realizability could neither be invalidated by the Proposition nor confirmed by the Conjecture; either because system (\ref{surjectivity_pol_system}) admits no solution on minimal clique covers or because even for such covers exhaustively checking for the absence of a solution  is not computationally feasible. We present the case-by-case analysis of the 9 graphs in {\ref{AppA}}. 

As the number of vertices considered increases, we expect the number of graphs left undecided by Proposition~\ref{observations_prop_cycles} as well as their diversity to increase quickly. This explains why we could not determine the number of realizable graphs on 8 vertices. Indeed, in all cases unconcerned by Proposition~\ref{observations_prop_cycles} (that represents $1662$ graphs on 8 vertices, of which 23 are realizable trees) for which no solution to system (\ref{surjectivity_pol_system}) can be found more ad-hoc arguments must be devised. The number of graphs in this situation is much larger than 9, not to mention that the arguments are both more difficult to find and more diverse than for the 9 graphs mentioned above.\\

We conclude this section by an observation about those graphs that are unrealizable. Since they are not, we might think that the elements of the corresponding trace monoids $\mathcal{T}$ cannot be drawn as cycles on a digraph and are therefore essentially different from walks and hikes. This however is incorrect.

\begin{proposition}\label{RealizSub}
Let $H$ be a graph that is unrealizable. Then there exist at least one realizable graph $H'$ such that $H$ is an induced subgraph of $H'$. Equivalently, let $\mathcal{T}$ be a trace monoid that is not a hike monoid. Then $\mathcal{T}$ is a submonoid of a hike monoid $\mathcal{H}'$.  
\end{proposition}

This indicates that elements of $\mathcal{T}$ are walks and walk-like objects after all, yet these are not drawable as such by themselves. Perhaps more strikingly, since $\mathcal{T}$ is a monoid in its own right it is \emph{algebraically closed}. Elements of $\mathcal{H}'\backslash \mathcal{T}$ are, in this sense, algebraically unrelated to those of $\mathcal{T}$. Yet allowing these additional elements turns the undrawable members of $\mathcal{T}$ into drawable objects.

\begin{proof}
 Given any graph $H$, we can always draw simple cycles sharing vertices so as to produce the cycle dependencies  dictated by $H$. Since $H$ is unrealizable this procedure must have created additional simple cycles in the graph $G'$ just  drawn. This implies that $H$ is an induced subgraph of $H'=\phi(G')$. 
\end{proof}

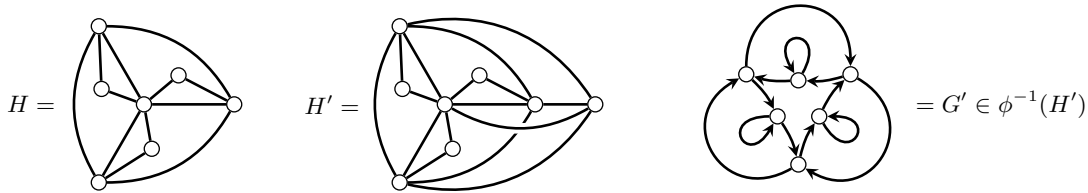
\begin{figure}[h!]
\vspace{-8mm}
    \centering
    \begin{tikzpicture}[Centering,line width=1]
    \def\radi{1.2}
    \def\radii{0.6}
    \def\alpha{40}
    \node[WhiteNode](c) at (0,0){};
    \node[WhiteNode](b1) at (\alpha:\radii){};
    \node[WhiteNode](b2) at (120+\alpha:\radii){};
    \node[WhiteNode](b3) at (240+\alpha:\radii){};
    \node[WhiteNode](a1) at (0:\radi){};
    \node[WhiteNode](a2) at (120:\radi){};
    \node[WhiteNode](a3) at (240:\radi){};
    \node[NullNode](l) at (-1.5,0){\small$H=$};
    \draw[-](c) to (a1);
    \draw[-](c) to (a2);
    \draw[-](c) to (a3);
    \draw[-](c) to (b1);
    \draw[-](c) to (b2);
    \draw[-](c) to (b3);
    \draw[-](a1) to (b1);
    \draw[-](a2) to (b2);
    \draw[-](a3) to (b3);
    \path[-](a1) edge[bend right=30] (a2);
    \path[-](a2) edge[bend right=30] (a3);
    \path[-](a3) edge[bend right=30] (a1);
    \begin{scope}[shift={(4,0)}]
     \node[WhiteNode](c) at (0,0){};
    \node[WhiteNode](b1) at (\alpha:\radii){};
    \node[WhiteNode](b2) at (120+\alpha:\radii){};
    \node[WhiteNode](b3) at (240+\alpha:\radii){};
    \node[WhiteNode](a1) at (0:\radi){};
    \node[WhiteNode](a2) at (120:\radi){};
    \node[WhiteNode](a3) at (240:\radi){};
    \node[WhiteNode](d) at (2,0){};
    \node[NullNode](l) at (-1.5,0){\small$H'=$};
    \path[-](a1) edge[bend right=30] (a2);
    \path[-](a3) edge[bend right=30] (a1);
    \path[-,line width=4,color=white](d) edge[bend left=30] (c);
    \draw[-](c) to (a1);
    \draw[-](c) to (a2);
    \draw[-](c) to (a3);
    \draw[-](c) to (b1);
    \draw[-](c) to (b2);
    \draw[-](c) to (b3);
    \draw[-](a1) to (b1);
    \draw[-](a2) to (b2);
    \draw[-](a3) to (b3);
    \draw[-](d) to (a1);
    \path[-](a2) edge[bend right=30] (a3);
    \path[-](d) edge[bend left=30] (c);
    \path[-](d) edge[bend right=35] (a2);
    \path[-](d) edge[bend left=35] (a3);
    \end{scope}
    \begin{scope}[shift={(8.7,0)}]
    \def\radi{0.8}
    \def\radii{0.4*\radi}
    \node[WhiteNode] (a) at (30:\radi){};
    \node[WhiteNode] (b) at (150:\radi){};
    \node[WhiteNode] (c) at (270:\radi){};
    \node[WhiteNode] (d) at (330:\radii){};
    \node[WhiteNode] (e) at (90:\radii){};
    \node[WhiteNode] (f) at (210:\radii){};
    \path[->](a) edge[bend left=10] (e);
    \path[->](e) edge[bend left=10] (b);
    \path[->](b) edge[bend left=10] (f);
    \path[->](f) edge[bend left=10] (c);
    \path[->](c) edge[bend left=10] (d);
    \path[->](d) edge[bend left=10] (a);
    \path[<-](a) edge[in=90,out=90,distance=30] (b);
    \path[<-](b) edge[in=210,out=210,distance=30] (c);
    \path[<-](c) edge[in=330,out=330,distance=30] (a);
    \path[<-](d) edge[in=300,out=360,distance=20] (d);
    \path[<-](e) edge[in=60,out=120,distance=20] (e);
    \path[<-](f) edge[in=180,out=240,distance=20] (f);
    \node[NullNode](l) at (2.7,0){\small$=G'\in\phi^{-1}(H')$};
    \end{scope}
    \end{tikzpicture}
    \vspace{-3mm}
    \caption{An example of Proposition~\ref{RealizSub}, graph $H$ is unrealizable yet appears as an induced subgraph of $H'$ which is realized by the digraph $G'$ shown on the  right.}
\end{figure}

\begin{remark}\label{remrealizsub}
Proposition~\ref{RealizSub} generalizes our preceding observation in (\ref{EqSquare}) about the unrealizable square graph to all unrealizable graphs. As a corollary there cannot be a forbidden induced-subgraph criterion for realizability. 
\end{remark}

\section{Injectivity: graphs with isomorphic hike monoids}
\subsection{Transformations between $\phi$-equivalent graphs}
We now turn to question~2 of \S\ref{section_observations}, that is we seek to find a description of $\phi$-equivalence.
{We first present two sufficient conditions for two digraphs to be $\phi$-equivalent, then show how these fail to be necessary on examples we found to be particularly telling of the difficulty of describing $\phi$-equivalence classes.}

{We start with} three digraph transformations which preserve the hike monoid. Let $G=(V,E)$ be a digraph. For $(v_1,v_2)$, an edge of $G$, we call {\em reversing} $(v_1,v_2)$ the transformation consisting in replacing the edge $(v_1,v_2)$ by the edge $(v_2,v_1)$. 
For $v$ a vertex of $G$, we call {\em jumping} the vertex $v$ the transformation which consists in removing $v$ from $V$ and adding an edge $(v_1, v_2)$ for each pair of edges $(v_1,v)$ and $(v,v_2)$.
See below for an example of jumping at the gray vertex:


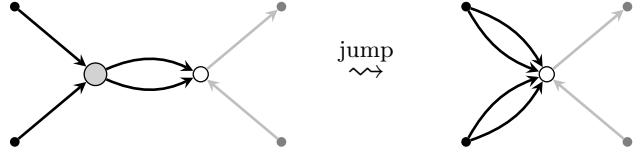
\begin{equation*}
    \begin{tikzpicture}[line width=1,Centering]
    \def\radi{1.4}
    \node[GrayNode] (a) at (0,0){};
    \node[WhiteNode] (b) at (\radi,0){};
    \node[BlackPointNode] (c) at (140:\radi){};
    \node[BlackPointNode] (d) at (220:\radi){};
    \node[GrayPointNode] (e) at ($(\radi,0)+(40:\radi)$){};
    \node[GrayPointNode] (f) at ($(\radi,0)+(-40:\radi)$){};
    \draw[->] (c) to (a);
    \draw[->] (d) to (a);
    \path[->] (a) edge[bend left=25] (b);
    \path[->] (a) edge[bend right=25] (b);
    \draw[->,lightgray] (b) to (e);
    \draw[->,lightgray] (f) to (b);
    \begin{scope}[shift={(6,0)}]
    \node[WhiteNode] (a) at (0,0){};
    \node[BlackPointNode] (c) at (140:\radi){};
    \node[BlackPointNode] (d) at (220:\radi){};
    \node[GrayPointNode] (e) at (40:\radi){};
    \node[GrayPointNode] (f) at (-40:\radi){};
    \path[->] (c) edge[bend left=20] (a);
    \path[->] (c) edge[bend right=20] (a);
    \path[->] (d) edge[bend left=20] (a);
    \path[->] (d) edge[bend right=20] (a);
    \draw[->,lightgray] (a) to (e);
    \draw[->,lightgray] (f) to (a);
    \end{scope}
    \node[NullNode] (l) at (3.6,0){\Large$\rightsquigarrow$};
    \node[NullNode] (m) at (3.6,0.3){\small jump};
    \end{tikzpicture}
\end{equation*}

\begin{proposition}\label{injectivity_transfo}
    Let $G$ be a digraph.
    \begin{enumerate}
        \item Reversing all the edges of $G$ preserves the hike monoid of $G$.
        \item \label{injectivity_transfo2}Let $v_1,v_2$ be two vertices of $G$ and all outgoing edges from $v_1$ point to $v_2$. Then jumping $v_1$ preserves the hike monoid of $G$.
        \item \label{injectivity_transfo3}Let $v_1,v_2$ be two vertices of $G$ and all in-going edges to $v_2$ come from $v_1$. Then jumping $v_2$ preserves the hike monoid of $G$.
    \end{enumerate}
\end{proposition}

\begin{proof}~
    \begin{enumerate}
        \item Let $G'$ be the digraph obtained by reversing all the edges of $G$. Then there is a simple cycle $(v_1,\dots, v_n)$ in $G'$ if and only if there is a simple cycle
        $(v_n,\dots, v_1)$ in $G$, hence $G\sim_{\phi} G'$.
        \item Let $G'$ be the digraph obtained by jumping $v_1$. Since the only outgoing edge from $v_1$ are of the form $(v_1,v_2)$, every simple cycle of $G$ passing by $v_1$ must also 
        pass by $v_2$. Hence there is a bijection between the simple cycles of $G$ and $G'$ which sends simple cycles without $v_1$ on themselves and simple cycles with $v_1$ on the
        simple cycle obtained by removing $v_1$. Hence $G\sim_{\phi} G'$.
        \item The last item is a direct consequences of the previous two.
    \end{enumerate}
\end{proof}

Proposition~\ref{injectivity_transfo} permits a reduction of the problem of realizability to cubic graphs, that is digraphs where all vertices have total degree 3:

\begin{corollary}
\label{C:cubic_is_sufficient}
    A graph is realizable if and only if it is realizable by a cubic graph.
\end{corollary}

\begin{proof}
    Let $H$ be a connected realizable graph and let $G=(V,E)$ be a strongly connected digraph realizing $H$ and $v$ be a vertex of $G$. Since $G$ is strongly connected, $v$ can not
    have a total degree $d(v)$ equal to $1$ and $d(v)=2$ if and only if $v$ has one in-going edge and one out-going edge. In this case, from Proposition~\ref{injectivity_transfo}, we
    know that we can jump $v$ while conserving $H$. Suppose now $d(v)>3$ and respectively denote by $u_0,\dots, u_{n-1}$ and $v_1,\dots, v_k$ the vertices in the sets 
    $\set{u\in V\,|\, (u,v)\in E}$ and $\set{u\in V\,|\, (v,u)\in E}$. We define the digraph $G'$ by transforming $G$ has follow (see Figure~\ref{F:cubic_is_sufficient} for an example):
    \begin{itemize}
        \item remove the vertex $v$ from $G$,
        \item add $n+k-2$ vertices $w_1,\dots,w_{n+k-2}$ and the edges $(w_i,w_{i+1})$ for $1\leq i\leq n+k-3$,
        \item add edges $(u_0,w_1)$, $(u_i,w_i)$ for $1\leq i\leq n-1$, $(w_{n+i-1}, v_i)$ for $1\leq i \leq k-1$ and $(w_{n+k-2},v_k)$.
    \end{itemize}
    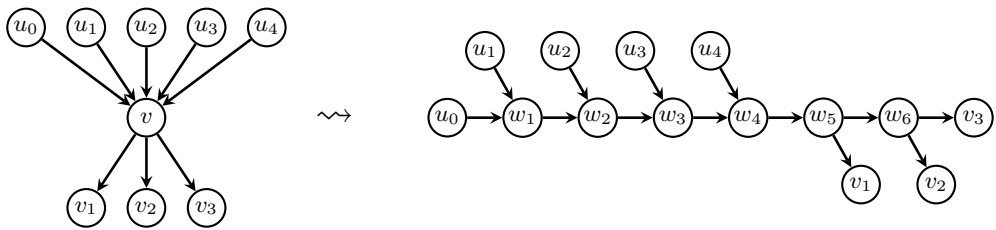
\begin{figure}[t!]
    \centering
    \begin{tikzpicture}[line width=1,Centering]
    \def\radi{0.8}
     \node[WhiteLabelNode](u0) at (-2*\radi,1.5*\radi){\small$u_0$};
     \node[WhiteLabelNode](u1) at (-\radi,1.5*\radi){\small$u_1$};
     \node[WhiteLabelNode](u2) at (0,1.5*\radi){\small$u_2$};
     \node[WhiteLabelNode](u3) at (\radi,1.5*\radi){\small$u_3$};
     \node[WhiteLabelNode](u4) at (2*\radi,1.5*\radi){\small$u_4$};
     \node[WhiteLabelNode](v) at (0,0){\small$v$};
     \node[WhiteLabelNode](v1) at (-\radi,-1.5*\radi){\small$v_1$};
     \node[WhiteLabelNode](v2) at (0,-1.5*\radi){\small$v_2$};
     \node[WhiteLabelNode](v3) at (\radi,-1.5*\radi){\small$v_3$};
    \draw[->](u0) to (v); 
    \draw[->](u1) to (v); 
    \draw[->](u2) to (v); 
    \draw[->](u3) to (v); 
    \draw[->](u4) to (v); 
    \draw[->](v) to (v1); 
    \draw[->](v) to (v2); 
    \draw[->](v) to (v3); 
    \node[NullNode](asso) at (2.5,0){\Large$\rightsquigarrow$};
    \begin{scope}[shift={(6,0)}]
    \def\radi{1}
    \node[WhiteLabelNode](u0) at (-2*\radi,0){\small$u_0$};
    \node[WhiteLabelNode](u1) at (-1.5*\radi,0.89*\radi){\small$u_1$};
    \node[WhiteLabelNode](u2) at (-0.5*\radi,0.89*\radi){\small$u_2$};
    \node[WhiteLabelNode](u3) at (0.5*\radi,0.89*\radi){\small$u_3$};
    \node[WhiteLabelNode](u4) at (1.5*\radi,0.89*\radi){\small$u_4$};
    \node[WhiteLabelNode](w1) at (-\radi,0){\small$w_1$};
    \node[WhiteLabelNode](w2) at (0,0){\small$w_2$};
    \node[WhiteLabelNode](w3) at (\radi,0){\small$w_3$};
    \node[WhiteLabelNode](w4) at (2*\radi,0){\small$w_4$};
    \node[WhiteLabelNode](w5) at (3*\radi,0){\small$w_5$};\node[WhiteLabelNode](w6) at (4*\radi,0){\small$w_6$};
    \node[WhiteLabelNode](v1) at (3.5*\radi,-0.89*\radi){\small$v_1$};
    \node[WhiteLabelNode](v2) at (4.5*\radi,-0.89*\radi){\small$v_2$};
    \node[WhiteLabelNode](v3) at (5*\radi,0){\small$v_3$};
    \draw[->](u0) to (w1);
    \draw[->](u1) to (w1);
    \draw[->](u2) to (w2);
    \draw[->](u3) to (w3);
    \draw[->](u4) to (w4);
    \draw[->](w5) to (v1);
    \draw[->](w6) to (v2);
    \draw[->](w6) to (v3);
    \draw[->](w1) to (w2);
    \draw[->](w2) to (w3);
    \draw[->](w3) to (w4);
    \draw[->](w4) to (w5);
    \draw[->](w5) to (w6);
    \end{scope}
    \end{tikzpicture}
    \caption{How to transform vertices and their adjacent edges so as to form a cubic graph: an illustration of Corollary~\ref{C:cubic_is_sufficient} with $n=4$ and $k=3$.}
    \label{F:cubic_is_sufficient}
\end{figure}
    By construction, vertices $w_1,\dots, w_{n+k-2}$ have total degree 3 and we can recover digraph $G$ by jumping them. Since by Proposition~\ref{injectivity_transfo} this
    preserves the hike monoid, $G'$ also realizes $H$.
    By applying the same transformation to every vertices of $G$, we obtain a cubic graph.
\end{proof}

Proposition~\ref{injectivity_transfo} give us a potential characterisations of 
$\phi$-equivalence: $G_1\sim_{\phi}G_2$ if and only if $G_1$ and $G_2$ reduce to the same digraph after recursively applying the transformations of Proposition~\ref{injectivity_transfo} until we obtain a
digraph where each vertex has in-degree and out-degree at least $2$. The digraph so obtained would then be a canonical representative of the $\phi$-equivalent class $\phi^{-1}(H)$ to which $G_1$ and $G_2$ belong. An alternative canonical representative would be a cubic graph realizing $H$, the existence of which is guaranteed by  Corollary~\ref{C:cubic_is_sufficient}. For example, all complete graphs on $n$ vertices $H=K_n$ are realizable, being realized by the bouquet graph $B_n$ with $n$ self-loops on a single vertex, itself irreducible under Proposition~\ref{injectivity_transfo}. Alternatively, the cubic ladder digraph $L_n$
 
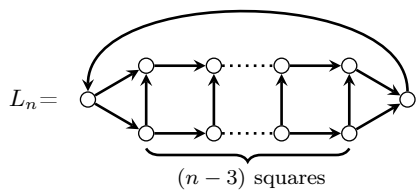
\vspace{-4mm}
    \begin{equation*}
        \begin{tikzpicture}[line width=1,Centering]
        \def\radi{0.9}
        \node[WhiteNode](b1) at (\radi,0){}; 
        \node[WhiteNode](b2) at (2*\radi,0){}; 
        \node[WhiteNode](b3) at (3*\radi,0){}; 
        \node[WhiteNode](b4) at (4*\radi,0){}; 
        \node[WhiteNode](t1) at (\radi,\radi){}; 
        \node[WhiteNode](t2) at (2*\radi,\radi){}; 
        \node[WhiteNode](t3) at (3*\radi,\radi){}; 
        \node[WhiteNode](t4) at (4*\radi,\radi){}; 
        \node[WhiteNode](l) at (\radi-0.86*\radi,0.5*\radi){}; 
        \node[WhiteNode](r) at (4*\radi+0.86*\radi,0.5*\radi){}; 
        \draw[->](b1) to (b2);
        \draw[dotted](b2) to (b3);
        \draw[->](b3) to (b4);
        \draw[->](t1) to (t2);
        \draw[dotted](t2) to (t3);
        \draw[->](t3) to (t4);
        \draw[->](b1) to (t1);
        \draw[->](b2) to (t2);
        \draw[->](b3) to (t3);
        \draw[->](b4) to (t4);
        \draw[->](l) to (b1);
        \draw[->](l) to (t1);
        \draw[->](b4) to (r);
        \draw[->](t4) to (r);
        \path[->](r) edge[out=90,in=90,distance=40] (l);
        \draw[decorate,decoration={brace,amplitude=5pt}] (4*\radi,-0.2) -- (\radi,-0.2);
        \node[NullNode](brace) at (2*\radi+0.5,-0.6){\footnotesize $(n-3)$ squares};
        \node[NullNode] at (-0.6,0.45){\small$L_n$=};
        \end{tikzpicture}
    \vspace{-6mm}
    \end{equation*}
satisfies $\phi(L_n)=K_n$ for $n\geq 3$, and $L_n$ reduces to $B_n$ under Proposition~\ref{injectivity_transfo}.

 Unfortunately while it is true that two digraphs reducing to the same digraph are necessarily $\phi$-equivalent, the converse proposition does not hold. For example the following
three digraphs are realizations of the complete graph on $17$ vertices $K_{17}$, yet do not reduce to the same digraph under the transformations listed above
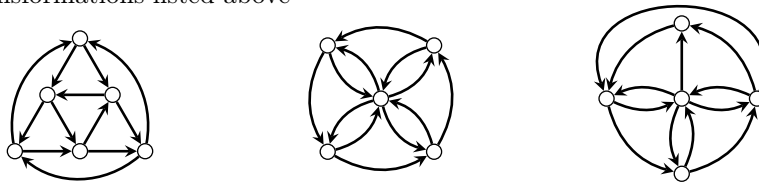
\vspace{-7mm}
\begin{equation*}
    \begin{tikzpicture}[line width=1,Centering]
    \def\radi{0.5}
    \node[WhiteNode](a) at (30:\radi){};
    \node[WhiteNode](b) at (150:\radi){};
    \node[WhiteNode](c) at (270:\radi){};
    \node[WhiteNode](d) at (90:2*\radi){};
    \node[WhiteNode](e) at (210:2*\radi){};
    \node[WhiteNode](f) at (330:2*\radi){};
    \draw[->](a) to (b);
    \draw[->](b) to (c);
    \draw[->](c) to (a);
    \draw[->](d) to (a);
    \draw[->](d) to (b);
    \draw[->](b) to (e);
    \draw[->](e) to (c);
    \draw[->](c) to (f);
    \draw[->](a) to (f);
    \path[->](f) edge[bend right=40] (d);
    \path[->](f) edge[bend left=40] (e);
    \path[->](e) edge[bend left=40] (d);
    \begin{scope}[shift={(4,0.2)}]
    \def\radi{1}
    \node[WhiteNode](o) at (0,0){};
    \node[WhiteNode](a) at (45:\radi){};
    \node[WhiteNode](b) at (135:\radi){};
    \node[WhiteNode](c) at (225:\radi){};
    \node[WhiteNode](d) at (315:\radi){};
    \path[->](o) edge[bend right=30] (a);
    \path[->](a) edge[bend right=30] (o);
    \path[->](o) edge[bend right=30] (b);
    \path[->](b) edge[bend right=30] (o);
    \path[->](o) edge[bend right=30] (c);
    \path[->](c) edge[bend right=30] (o);
    \path[->](o) edge[bend right=30] (d);
    \path[->](d) edge[bend right=30] (o);
    \path[->](a) edge[bend right=30] (b);
    \path[->](b) edge[bend right=30] (c);
    \path[->](c) edge[bend right=30] (d);
    \path[->](d) edge[bend right=30] (a);
    \end{scope}
    \begin{scope}[shift={(8,0.2)}]
    \def\radi{1}
    \node[WhiteNode] (o) at (0,0){};
    \node[WhiteNode] (a1) at (-90:\radi){};
    \node[WhiteNode] (a2) at (0:\radi){};
    \node[WhiteNode] (a3) at (90:\radi){};
    \node[WhiteNode] (a4) at (180:\radi){};
    \draw[->](o) to (a3);
    \path[->](o) edge[bend right=25] (a2);
    \path[->](o) edge[bend right=25] (a1);
    \path[->](o) edge[bend right=25] (a4);
    \path[->](a2) edge[bend right=25] (o);
    \path[->](a4) edge[bend right=25] (o);
    \path[->](a1) edge[bend right=25] (o);
    \path[->](a1) edge[bend right=30] (a2);
    \path[->](a2) edge[bend right=30] (a3);
    \path[->](a3) edge[bend right=30] (a4);
    \path[->](a4) edge[bend right=30] (a1);
    \path[->](a2) edge[in=110,out=70,distance=45] (a4);
    \end{scope}
    \end{tikzpicture}
\end{equation*}
More generally we found that for most graph theoretic properties, we could construct a realization of $K_n$ that has  the property and another that does not have it. For example both of the following digraphs realize $K_{12}$
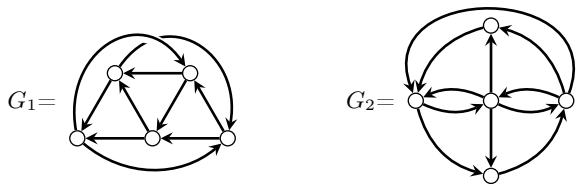
\vspace{-5mm}
\begin{equation*}
    \begin{tikzpicture}[line width=1,Centering]
    \begin{scope}[shift={(-.5,0)}]
    \def\radi{1}
    \node[WhiteNode](a) at (0:\radi){}; 
    \node[WhiteNode](b) at (0,0){}; 
    \node[WhiteNode](c) at (180:\radi){}; 
    \node[WhiteNode](d) at (60:\radi){}; 
    \node[WhiteNode](e) at (120:\radi){}; 
    \draw[->](a) to (d);
    \draw[->](a) to (b);
    \draw[->](d) to (b);
    \draw[->](d) to (e);
    \draw[->](b) to (c);
    \draw[->](e) to (c);
    \draw[->](b) to (e);
    \path[->](c) edge[bend right=40] (a);
    \path[<-](a) edge[out=80,in=55,distance=30] (e);
    \path[-,line width=4,color=white](c) edge[out=100,in=125,distance=30] (d);
    \path[->](c) edge[out=100,in=125,distance=30] (d);
    \node[NullNode] at (-1.6,0.5){\small$G_{1}$=};
    \end{scope}
    \begin{scope}[shift={(4,.5)}]
     \def\radi{1}
     \node[WhiteNode] (o) at (0,0){};
     \node[WhiteNode] (a1) at (-90:\radi){};
     \node[WhiteNode] (a2) at (0:\radi){};
     \node[WhiteNode] (a3) at (90:\radi){};
     \node[WhiteNode] (a4) at (180:\radi){};
     \draw[->](o) to (a1);
     \draw[->](o) to (a3);
     \path[->](o) edge[bend right=25] (a2);
     \path[->](o) edge[bend right=25] (a4);
     \path[->](a2) edge[bend right=25] (o);
     \path[->](a4) edge[bend right=25] (o);
     \path[->](a1) edge[bend right=30] (a2);
     \path[->](a2) edge[bend right=30] (a3);
     \path[->](a3) edge[bend right=30] (a4);
     \path[->](a4) edge[bend right=30] (a1);
     \path[->](a2) edge[in=110,out=70,distance=45] (a4);
     \node[NullNode] at (-1.6,0){\small$G_{2}$=};
     \end{scope}
    \end{tikzpicture}
\end{equation*}
yet $G_1$ is vertex-transitive while $G_2$ is not, $G_2$ is planar and $G_1$ is not; or consider the flower graph $F_n$ on $n$ petals
 
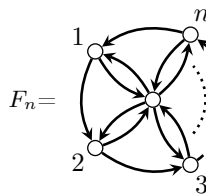
\vspace{-7mm}
\begin{equation*}
\begin{tikzpicture}[line width=1,Centering]
\def\radi{1}
\node[WhiteNode] (o) at (0,0){};
\node[WhiteNode] (a) at (60:\radi){};
\node[WhiteNode] (b) at (140:\radi){};
\node[WhiteNode] (c) at (220:\radi){};
\node[WhiteNode] (d) at (300:\radi){};
\node[NullNode] (la) at (60:1.3*\radi){$n$};
\node[NullNode] (lb) at (140:1.3*\radi){$1$};
\node[NullNode] (lc) at (220:1.3*\radi){$2$};
\node[NullNode] (ld) at (300:1.3*\radi){$3$};
\path[->](a) edge[bend right=25] (b);
\path[->](b) edge[bend right=25] (c);
\path[->](c) edge[bend right=25] (d);
\path[->,dashed](d) edge[bend right=55] (a);
\path[->](o) edge[bend right=20] (a);
\path[->](a) edge[bend right=20] (o);
\path[->](o) edge[bend right=20] (b);
\path[->](b) edge[bend right=20] (o);
\path[->](o) edge[bend right=20] (c);
\path[->](c) edge[bend right=20] (o);
\path[->](o) edge[bend right=20] (d);
\path[->](d) edge[bend right=20] (o);
\draw[dotted] (320:0.7*\radi) arc (-40:40:0.7*\radi);
 \node[NullNode] at (-1.6,0){\small$F_{n}$=};
\end{tikzpicture}
\vspace{-5mm}
\end{equation*}
for which $\phi(F_n)=K_{n^2+1}$. Flower $F_n$ has a highly skewed degree distribution while that of the ladder $L_{n^2+1}$ is uniform; etc. This means that even considering the purportedly simpler case of the complete graphs $H=K_n$, we still find a profusion of  $\phi$-equivalent graphs unaccounted for by Proposition~\ref{injectivity_transfo} which moreover do not seem to be related to one another in any obvious way.
    
These observations show that to pursue a full characterisation of $\phi$-equivalence we need an exhaustive description of the hike dependency graphs of all graphs with in- and out-degrees at least $2$. We could not find such a description--and a plethora of cases make this task clearly daunting. For a realizable graph $H$ we do however have a lower bound and two upper bounds on the number of vertices and edges that are satisfied by at least one graph in $\phi^{-1}(H)$ {which provide some constraint on $\phi^{-1}(H)$. To describe these bounds consider $G=(V_G,E_G)$ a strongly connected digraph with in- and out-degrees at least $2$ and $H=(V_H,E_H)=\phi(G)$.}

First, as remarked in the previous section, every vertex of $G$ correspond to a clique in $H$, consequently $\card{V_G}$ must be at least the size of a minimal clique cover of $H$. 

Second, to give upper bounds on $\card{V_G}$ and $\card{E_G}$, we begin by briefly recalling the notion of ear decomposition of a graph while
we refer the reader to \cite{BJG} for a thorough discussion of this topic.
An {\em ear decomposition} of $G$ is a sequence $(G_1,\dots,G_n)$ of 
directed subgraphs of $G$, called {\em ears}, such that:
\begin{itemize}
    \item $G_1$ is a simple cycle,
    \item for $i>1$, $G_i$ is a path (possibly closed) with exactly its start and end vertices already present among the vertices of $G_j$ for $1\leq j<i$,
    \item $G = \bigcup G_i$.
\end{itemize}
\noindent Consider the procedure of constructing $G$ by starting with $\tilde{G}_1 = G_1$ and recursively constructing $\tilde{G}_i$ as 
$\tilde{G}_i = \tilde{G}_{i-1}\cup G_i$ for $i>1$. 
Since the first ear has as many vertices as edges and that the addition of every other ear adds exactly one more edge than vertices, the length of an ear decomposition is always equal to
$\card{E_G}-\card{V_G}+1$. Now remark that graphs $\tilde{G}_i$ are all strongly connected so that adding $G_{i+1}$ to $\tilde{G}_i$ forms at least one
additional cycle. Hence we have the inequality $\card{V_H} \geq \card{E_G} - \card{V_G} +1$. Since furthermore we can always reduce $G$ to have in- and out-degrees at least $2$, we
have that $\card{E_G} = \frac{1}{2}\sum_{v\in V_G} d_{in}(v)+d_{out}(v) \geq \frac{4}{2}\card{V_G} = 2\card{V_G}$ giving the upper bounds 
$\card{V_H} -1 \geq \card{V_G}$ and 
$2\card{V_H}-2\geq \card{E_G}$.

\subsection{Implications of sharing the same dependency graph}
\subsubsection{Simple graphs do not relate well with their simple cycles}
Let $H$ be a realizable graph. As we have seen so far, digraphs realizing $H$ may be very diverse, not even sharing basic graph theoretic properties such as regularity. By contrast, in this section we present some properties they must necessarily share by virtue of the isomorphism between their hike monoids. The first result on this matter was obtained in \cite{GiscardRochet2016} by regarding graphs as bidirected digraphs. 

\begin{theorem}[Giscard and Rochet, 2016 \cite{GiscardRochet2016}]\label{poset_charac} Let $G$ be a connected graph with no self-loop. Then the hike dependency graph $H$ of $G$ determines $G$ uniquely up to isomorphism, unless $G \in \{ K_3,K_{1,5} \}$. 
\end{theorem}

{While this result seems to be quite powerful at first glance, the seemingly harmless assumption that the graph is bidirected (i.e., effectively undirected) is essential. Should even a single edge be directed the theorem would completely fail to hold. Indeed,} the proof relies exclusively on simple cycles of length 2 (called backtracks) sustained by every undirected edge $\{i,j\}$ on $G$. Vertices of $H$ corresponding to backtracks in $G$ can always be identified even though the lengths of the simple cycles corresponding to the vertices of $H$ are not known (with the single exception $H=K_5$ leading to the $\{K_3,K_{1,5}\}$ special case). Now the induced subgraph of $H$ formed by the backtracks is the line graph $L(G)$ of $G$, which is well known to determine $G$ uniquely. In other terms, the theorem makes no use of longer simple cycles, so much so that if the information regarding the backtracks is discarded from $H$, that is we only have access to $\phi(G)\backslash L(G)$, then graphs suffer from exactly the same shortfalls as directed graphs when it comes to relating them to their closed walks. That is, 
\begin{itemize}
    \item There are graphs $G$ and $G'$ with $\phi(G)/L(G) = \phi(G')/L(G')$ and $G$ and $G'$ do not share fundamental graph theoretic properties. See Figure~\ref{F:line_graph} below.
    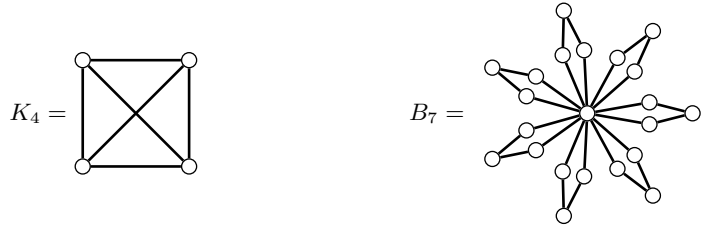
\begin{figure}[h!]
    \centering
    \begin{tikzpicture}[Centering,line width=1]
        \def\radi{1}
        \node[WhiteNode](1) at (45:\radi){};
        \node[WhiteNode](2) at (135:\radi){};
        \node[WhiteNode](3) at (225:\radi){};
        \node[WhiteNode](4) at (315:\radi){};
        \node[NullNode](l) at (-1.3,0){\small$K_4=$};
        \draw(1) to (2);
        \draw(1) to (3);
        \draw(1) to (4);
        \draw(2) to (3);
        \draw(2) to (4);
        \draw(3) to (4);
        \begin{scope}[shift={(6,0)}]
        \def\alpha{360/7}
        \def\radi{1.4}
        \def\radii{0.6*\radi}
        \def\beta{10}
        \node[WhiteNode](o) at (0,0){};
        \node[WhiteNode](a1) at (0:\radi){};
        \node[WhiteNode](a2) at (\alpha:\radi){};
        \node[WhiteNode](a3) at (2*\alpha:\radi){};
        \node[WhiteNode](a4) at (3*\alpha:\radi){};
        \node[WhiteNode](a5) at (4*\alpha:\radi){};
        \node[WhiteNode](a6) at (5*\alpha:\radi){};
        \node[WhiteNode](a7) at (6*\alpha:\radi){};
        \node[WhiteNode](b1) at (\beta:\radii){};
        \node[WhiteNode](b2) at (\alpha+\beta:\radii){};
        \node[WhiteNode](b3) at (2*\alpha+\beta:\radii){};
        \node[WhiteNode](b4) at (3*\alpha+\beta:\radii){};
        \node[WhiteNode](b5) at (4*\alpha+\beta:\radii){};
        \node[WhiteNode](b6) at (5*\alpha+\beta:\radii){};
        \node[WhiteNode](b7) at (6*\alpha+\beta:\radii){};
        \node[WhiteNode](c1) at (-\beta:\radii){};
        \node[WhiteNode](c2) at (\alpha-\beta:\radii){};
        \node[WhiteNode](c3) at (2*\alpha-\beta:\radii){};
        \node[WhiteNode](c4) at (3*\alpha-\beta:\radii){};
        \node[WhiteNode](c5) at (4*\alpha-\beta:\radii){};
        \node[WhiteNode](c6) at (5*\alpha-\beta:\radii){};
        \node[WhiteNode](c7) at (6*\alpha-\beta:\radii){};
        \draw(o) to (b1);
        \draw(o) to (b2);
        \draw(o) to (b3);
        \draw(o) to (b4);
        \draw(o) to (b5);
        \draw(o) to (b6);
        \draw(o) to (b7);
        \draw(o) to (c1);
        \draw(o) to (c2);
        \draw(o) to (c3);
        \draw(o) to (c4);
        \draw(o) to (c5);
        \draw(o) to (c6);
        \draw(o) to (c7);
        \draw(b1) to (a1);
        \draw(b2) to (a2);
        \draw(b3) to (a3);
        \draw(b4) to (a4);
        \draw(b5) to (a5);
        \draw(b6) to (a6);
        \draw(b7) to (a7);
        \draw(c1) to (a1);
        \draw(c2) to (a2);
        \draw(c3) to (a3);
        \draw(c4) to (a4);
        \draw(c5) to (a5);
        \draw(c6) to (a6);
        \draw(c7) to (a7);
        \node[NullNode](l) at (-2,0){\small $B_7=$};
        \end{scope}
        \end{tikzpicture}
    \caption{A pair of graphs with $\phi(K_4)/L(K_4) = \phi(B_7)/L(B_7)$, graph $K_4$ on the left is Hamiltonian and vertex-transitive but not bipartite unlike the bouquet graph $B_7$ on the right, which is bipartite and neither Hamiltonian nor vertex-transitive.}
    \label{F:line_graph}
    \end{figure}
    \item Given a graph $H$, deciding whether there exists a graph $G$ such that $H=\phi(G)/L(G)$ is widely open. 
\end{itemize}
These issues are pressing since many tools of modern network analysis and algebraic graph theory explicitly disregard or forbid backtracks. This includes backtrackless walks methods \cite{Aziz2013}, Ihara zeta function and primitive orbits \cite{Terras2011} and open problems about graphs and cycles, foremost among which is the cycle double cover conjecture \cite{Jaeger1985}. 

\subsubsection{Invariants of hike monoids}
Recall from \S\ref{subsec:monoids} that for a digraph $G$, $H=\phi(G)$ is the dependency graph of its hike monoid, which is the trace monoid generated by its simple cycles under the partial commutation rule that two cycles commute if and only if they share no vertex.

Given that isomorphisms of hike monoids implement profound changes on digraphs, it is worth investigating quantities left invariant by these isomorphisms. The first type of such invariants are algebraic quantities based on the adjacency matrix. 
\begin{proposition}\label{Invariants1}
Let $G$ and $G'$ be two $\phi$-equivalent digraphs and let $\mat{A}$ and $\mat{A}'$ be their adjacency matrices, respectively. 
Then:
\begin{enumerate}
  \item[1)] $\det(\mat{I}-\mat{A})=\det(\mat{I}'-\mat{A}')$,
  \item[2)] $\perm(\mat{I}+\mat{A})=\perm(\mat{I}'+\mat{A}')$,
\end{enumerate}
where $\mat{I}$ and $\mat{I}'$ are identity matrices of appropriate sizes. 
\end{proposition}

\begin{proof}
A self-avoiding hike is a hike $h=c_1\cdots c_k$ in which all simple cycles $c_j$ are vertex-disjoint, hence commute. Consequently, a hike is self-avoiding if and only if the set $\{c_1,\hdots,c_k\}$ forms a clique in the complement graph $H^c$ of the hike dependency graph. It follows that there is a bijection $\psi$ between the sets of self-avoiding hikes on two $\phi$-equivalent digraphs that preserves the number of simple cycles in each self-avoiding hike. From the definition of the matrix determinant in terms of permutations and the representation of permutations as disjoint cycles we have 
$$
\sum_{h~\text{s.a.}}(-1)^{\Omega(h)} = \det(\mat{I}-\mat{A}),
$$
where the sum runs over \underline{s}elf-\underline{a}voiding hikes, $\Omega(h)$ denotes the number of simple cycles in $h$ and thus $(-1)^{\Omega(h)}$ is the signature of the permutation with cycle decomposition $h$ {(see \cite{viennot1989heaps, cartier1969})}. The sum is left unchanged by the monoid isomorphism $\psi$, which entails assertion 1) {Indeed, self-avoiding hikes are cliques in the independency graph $H^c$ which is invariant under $\psi$ so both the sizes of these cliques (the number $\Omega(h)$) as well as their total number is preserved by $\psi$. This last observation implies  assertion 2) as 
$
\sum_{h~\text{s.a.}}1 =\perm(\mat{I}+\mat{A})
$.}
\end{proof}

\begin{remark}
{The list of algebraic invariants provided in Proposition~\ref{Invariants1} is \emph{not} exhaustive, many more invariants can for example be inferred from \cite{Giscard2018AHA}. To see this consider two algebraic invariants of hike monoids $a(G)$ and $b(G)$ defined on a digraph $G$ and all its induced subgraphs $g\prec G$. Then the induced subgraph convolution $(a\ast b)(G)=\sum_{g\prec G}a(g)b(G\backslash g)$ is also an algebraic invariant of hike monoids. In other terms if $G$ and $G'$ are two $\phi$-equivalent digraphs, then $a(G)=a(G')$, $b(G)=b(G')$ and $(a\ast b)(G)=(a \ast b)(G')$. As an instance of this we unify and generalise assertions 1) and 2) with $\text{perm}(\mathbf{I}+\mathbf{A})^{\ast k}$, the $k$th-induced subgraph convolution of the permanent with itself, which is therefore an invariant of hike monoids for all $k\in \mathbb{Z}$  (the Hopf algebraic relation $\text{perm}(\mathbf{I}+\mathbf{A})^{\ast-1}=\det(\mathbf{I}+\mathbf{A})$ gives meaning to the negative $\ast$ powers of the permanent \cite{Giscard2018AHA}).
We may verify this claim directly by writing \cite{Giscard2018AHA}
$
 \text{perm}(\mat{I}+\mat{A})^{\ast k} = \sum_{h~\text{s.a.}}k^{\Omega(h)}
$ and reasoning as in the proof of Proposition \ref{Invariants1}.}

If the isomorphism of hike monoids between $\mathcal{H}_G$ and $\mathcal{H}_{G'}$ is length preserving, then all algebraic quantities computable from their adjacency matrices are identical (e.g. degree sequence, graph spectra, permanental polynomial $\perm(\mat{I}-z\mat{A})$ etc).  Allowing for the addition of transient vertices it is always possible to construct a pair of digraphs  related by a length-preserving isomorphism of hike monoids from a pair of digraphs related by an isomorphism that does not preserve the length. This gives a systematic mean of generating cospectral pairs, although the digraphs so generated are much more similar than just cospectral (remark also that there are pairs of cospectral graphs do not share the same hike dependency graph). Below is an example of digraphs related by such a morphism:
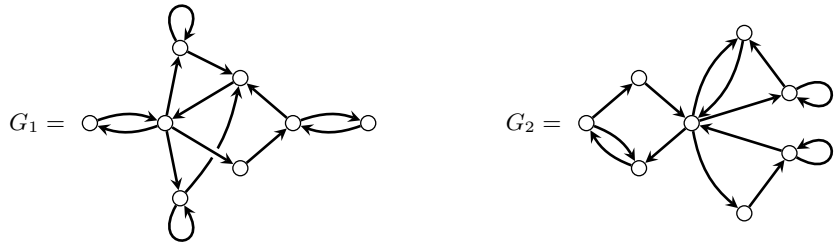
\vspace{-3mm}
\begin{equation*}
\begin{tikzpicture}[Centering,line width=1]
\def\radi{1.2}
\begin{scope}[shift={(-2,0)}]
\node[WhiteNode] (1) at (0,0){}; 
\node[WhiteNode] (2) at (1,0){}; 
\node[WhiteNode] (3) at (2,.6){}; 
\node[WhiteNode] (4) at (2,-.6){}; 
\node[WhiteNode] (5) at (1.2,1){}; 
\node[WhiteNode] (6) at (1.2,-1){}; 
\node[WhiteNode] (7) at (2.7,0){}; 
\node[WhiteNode] (8) at (3.7,0){}; 
\draw[->](2) to (5);
\draw[->](5) to (3);
\draw[->](3) to (2);
\draw[->](2) to (6);
\path[<-](3) edge[bend left=15] (6);
\draw[->](7) to (3);
\draw[->](4) to (7);
\path[->](5) edge[out=120,in=60,distance=20] (5);
\path[->](6) edge[out=-120,in=-60,distance=20] (6);
\path[->](7) edge[bend left=20] (8);
\path[->](8) edge[bend left=20] (7);
\path[->](1) edge[bend left=20] (2);
\path[->](2) edge[bend left=20] (1);
\draw[-,line width=4,color=white](2) to (4);
\draw[->](2) to (4);
\node[NullNode] (n) at (-.7,0){\small$G_1=$};
\end{scope}
\begin{scope}[shift={(6,0)}]
\def\radi{1.2}
\node[WhiteNode] (1) at (0,0){};
\node[WhiteNode] (2) at (.7,1.2){};
\node[WhiteNode] (3) at (1.3,.4){};
\node[WhiteNode] (4) at (.7,-1.2){};
\node[WhiteNode] (5) at (1.3,-.4){};
\node[WhiteNode] (6) at (-.7,-.6){};
\node[WhiteNode] (7) at (-1.4,0){};
\node[WhiteNode] (8) at (-.7,.6){};
\draw[->](1) to (3);
\draw[->](3) to (2);
\draw[->](4) to (5);
\draw[->](5) to (1);
\draw[->](8) to (1);
\draw[->](1) to (6);
\draw[->](7) to (8);
\path[->](1) edge[bend left=20] (2);
\path[->](2) edge[bend left=20] (1);
\path[<-](4) edge[bend left=20] (1);
\path[->](6) edge[bend left=20] (7);
\path[->](7) edge[bend left=20] (6);
\path[->](5) edge[out=-30,in=30,distance=20] (5);
\path[->](3) edge[out=30,in=-30,distance=20] (3);
\node[NullNode] (n) at (-2.1,0){\small$G_2=$};
\end{scope}
\end{tikzpicture}
\end{equation*}
\end{remark}

The second type of invariant are combinatorial invariants sensitive to how walks and hikes are composed of cycles:
\begin{proposition}\label{CombiInvar}
Let $G=(V_G,E_G)$ be a digraph, $\mathcal{H}_G$ its hike monoid and  $\mathcal{W}_G\subset \mathcal{H}_G$ be the set of walks in $\mathcal{H}_G$, that is the set of hikes $h=c_1\cdots c_k\in \mathcal{H}_G$ with a unique right simple cycle $c_k$. Let $\Omega(h)$
be the number of simple cycles in a hike $h$. 
Let $f_{\mathcal{H}_G}^{\Omega}(z):=\sum_{h\in \mathcal{H}_G} z^{\Omega(h)}$, 
$f_{\mathcal{W}_G}^{\Omega}(z):=\sum_{h\in \mathcal{W}_G} z^{\Omega(h)}$, 
be the associated ordinary generating functions on hikes and walks, 
respectively. \\ 
Then $G\sim_\Phi G'$ implies
\begin{enumerate}
  \item $f_{\mathcal{H}_G}^{\Omega}(z)=f_{\mathcal{H}_{G'}}^{\Omega}(z)$;
  \item $f_{\mathcal{W}_G}^{\Omega}(z)=f_{\mathcal{W}_{G'}}^{\Omega}(z)$.
  \end{enumerate}
  Now consider the set $\mathcal{W}_{G:i\to j}$ of rooted walks $w$ from vertex $i$ to vertex $j$ in $G$. Let $\Omega(w)$ designate the number of loops erased from $w$ following Lawler's loop erasing procedure \cite{lawler1987loop}. Let $\sigma(.)$ be an additive function on hikes.
  Then $G\sim_\Phi G'$ implies that for all $i,j\in V_G$ there exists $i',j'\in V_{G'}$ with
  \begin{enumerate}
   \setcounter{enumi}{2}
  \item $f_{\mathcal{W}_{G:i\to j}}^{\Omega}(z)=f_{\mathcal{W}_{G':i'\to j'}}^{\Omega}(z)$;
  \item The shape of the branched continued fraction representation of the generating series associated with an additive function $\sigma$ defined on rooted walks from $i$ to $j$ on $G$ is the same as that representing the generating function of $\sigma$ on rooted walks from $i'$ to $j'$ on $G'$.
\end{enumerate}
\end{proposition}

\begin{proof}
Claims 1. and 2. follow immediately from the definitions of $f_{\mathcal{H}_G}^{\Omega}(z)$ and $f_{\mathcal{W}_G}^{\Omega}(z)$ as sums over monoid elements. Since $\mathcal{H}_G$ and $\mathcal{H}_{G'}$ are isomorphic, these functions are the same on $G$ and $G'$ provided $\Omega(.)$ and $\omega(.)$ take on the same value over a hike in $\mathcal{H}_G$ and its counterpart under isomorphism in $\mathcal{H}_{G'}$, which is clearly true.

For assertions 3 and 4, consider $\sigma:\mathcal{W}_G \to \mathbb{R}$ an additive function on walks, i.e., for $w\in\mathcal{W}_G$ with $w=c_1\cdots c_k$ we have $\sigma(w) = \sum_{i=1}^k \sigma(c_i)$. Function $\sigma$ extends naturally to rooted walks so we can consider generating function 
$
f_{\mathcal{W}_{G:i\to j}}^{\sigma}(z):=\sum_{h\in \mathcal{W}_{G:i\to j}} z^{\sigma(h)}$.
Any such sum over all rooted walks has a unique representation as a branched continued fraction of finite depth and breadth called a path-sum \cite{giscard2012continued}. Each term of a path-sum corresponds to a rooted simple cycle on $G$. The continued fraction gains a branch if two or more rooted simple cycles share the same root. A branch gains depth when a simple cycle's root is located inside of another rooted simple cycle. A branch gains breadth when several simple cycles' roots are distinct and located inside of another simple cycle. Overall this implies that the shape of the path-sum is that of the spanning tree of $H$ whose level one vertices are the simple cycles crossing vertex $i$. Consequently, this shape is an invariant of isomorphisms of hike monoids since they preserve $H$. Assertion 3 follows from the peculiar choice $\sigma = \Omega$ and $\Omega(c)=1$ for any simple cycle $c\in\mathcal{C}$.
\end{proof}
\begin{example}
Consider the following two graphs with identical hike dependency graph:
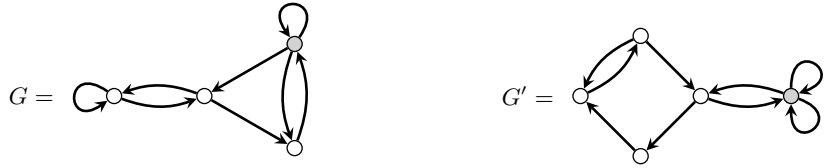
\vspace{-2mm}
\begin{equation*}
\begin{tikzpicture}[line width=1,Centering]
\def\radi{0.8}
\node[WhiteNode] (a) at (300:\radi){};
\node[MedGrayNode] (b) at (60:\radi){};
\node[WhiteNode] (c) at (180:\radi){};
\node[WhiteNode] (d) at (180:2.5*\radi){};
\path[->](a) edge[bend right=20] (b);
\path[->](b) edge[bend right=20] (a);
\draw[->](b) to (c);
\draw[->](c) to (a);
\path[->](c) edge[bend right=20] (d);
\path[->](d) edge[bend right=20] (c);
\path[->](d) edge[out=145,in=215,distance=20] (d);
\path[->](b) edge[out=55,in=125,distance=20] (b);
\node[NullNode] (n) at (-3.1,0){\small$G=$};
\begin{scope}[shift={(5,0)}]
\node[WhiteNode] (a) at (0:\radi){};
\node[WhiteNode] (b) at (90:\radi){};
\node[WhiteNode] (c) at (180:\radi){};
\node[WhiteNode] (d) at (270:\radi){};
\node[MedGrayNode] (e) at (2.5*\radi,0){};
\path[->](a) edge[bend right=20] (e);
\path[->](e) edge[bend right=20] (a);
\path[->](c) edge[bend right=20] (b);
\path[->](b) edge[bend right=20] (c);
\draw[->](b) to (a);
\draw[->](a) to (d);
\draw[->](d) to (c);
\path[->](e) edge[out=90,in=20,distance=20] (e);
\path[->](e) edge[out=340,in=270,distance=20] (e);
\node[NullNode] (n) at (-1.5,0){\small$G'=$};
\end{scope}
\end{tikzpicture}
\end{equation*}
Let $\sigma$ be the length function $\ell$, which is additive as necessary. Then the path-sum continued fraction representation of the ordinary generating function of all rooted closed walks from vertex $\GrayV$ to itself on $G$ and $G'$ associated with the length function is 
\begin{align*}
 \sum_{w\in\mathcal{W}_{G:\GrayVii\to \GrayVii}}z^{\ell(w)}&=\frac{1}{1-z-z^2-\frac{z^3}{1-\frac{z^2}{1-z}}}=(\mat{I}-z\mat{A}_{G})^{-1}_{\GrayVii\!\GrayVii},\\
 \sum_{w\in\mathcal{W}_{G':\GrayVii\to \GrayVii}}z^{\ell(w)}&=\frac{1}{1-z-z-\frac{z^2}{1-\frac{z^4}{1-z^2}}}=(\mat{I}'-z\mat{A}_{G'})^{-1}_{\GrayVii\!\GrayVii},
\end{align*}
where $\mat{A}_G$ and $\mat{A}_{G'}$ are the adjacency matrices of $G$ and $G'$, respectively. The common structure $T$ of both fractions is readily apparent and stems from the hike dependency graph $H=\phi(G)=\phi(G')$ as explained in the proof of Proposition~\ref{CombiInvar}:
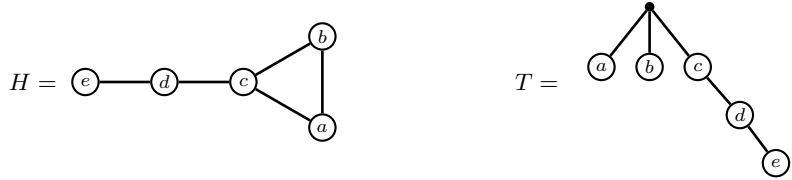
\vspace{-3mm}
\begin{equation*}
\begin{tikzpicture}[line width=1,Centering]
\def\radi{0.7}
\node[SmallWhiteLabelNode] (a) at (300:\radi){$a$};
\node[SmallWhiteLabelNode] (b) at (60:\radi){$b$};
\node[SmallWhiteLabelNode] (c) at (180:\radi){$c$};
\node[SmallWhiteLabelNode] (d) at (180:2.5*\radi){$d$};
\node[SmallWhiteLabelNode] (e) at (180:4*\radi){$e$};
\draw (a) to (b);
\draw (b) to (c);
\draw (c) to (a);
\draw (c) to (d);
\draw (d) to (e);
\node[NullNode] (n) at (-3.5,0){\small$H=$};
\begin{scope}[shift={(4.7,1)}]
\def\radi{0.8}
\node[SmallBlackNode] (r) at (0,0){};
\node[SmallWhiteLabelNode] (b) at (0,-\radi){$b$};
\node[SmallWhiteLabelNode] (a) at (-0.8*\radi,-\radi){$a$};
\node[SmallWhiteLabelNode] (c) at (0.8*\radi,-\radi){$c$};
\node[SmallWhiteLabelNode] (d) at (1.5*\radi,-1.8*\radi){$d$};
\node[SmallWhiteLabelNode] (e) at (2.1*\radi,-2.6*\radi){$e$};
\draw(r) to (a);
\draw(r) to (b);
\draw(r) to (c);
\draw(c) to (d);
\draw(d) to (e);
\node[NullNode] (n) at (-1.5,-1){\small$T=$};
\end{scope}
\end{tikzpicture}
\vspace{-1mm}
\end{equation*}

In $G$, $a$ and $e$ are self-loops, $b$ and $d$ are backtracks and $c$ is a triangle. In $G'$, $a$ and $b$ are self-loops, $c$ and $e$ are backtracks and $d$ is a square.
Since $\Omega$ is additive the same structure $T$  appears in $f_{\mathcal{W}_{G:1\to 1}}^{\Omega}(z)= f_{\mathcal{W}_{G':1\to 1}}^{\Omega}(z)$,
\begin{align*}
 &f_{\mathcal{W}_{G:1\to 1}}^{\Omega}(z)=f_{\mathcal{W}_{G':1\to 1}}^{\Omega}(z)=\frac{1}{1-z-z-\frac{z}{1-\frac{z}{1-z}}}.
\end{align*}
\end{example}

\section{Conclusion}
In this work we investigated the relation between digraphs and their walks and found it to be rather weak: no graph theoretic property seems to be directly related to the arrangement of simple cycles on a digraph. Among the graph  properties we found may be lost whilst leaving the hike monoidal structure of cycles invariant are vertex-transitivity, regularity, planarity, bipartiteness, (bi)directedness, Hamiltonicity, being Eulerian, being chordal, being triangle-free, chromatic number, graph spectra, in- and out-degree distributions and a majority of algebraic quantities computable from adjacency matrices. This list is undoubtedly non-exhaustive. As a corollary, there is a great variety of digraphs with isomorphic hike monoids and characterizing all transformations relating such digraphs remains completely open.

Conversely, just deciding which arrangements of simple cycles exist at all  is highly non-trivial, even allowing for multidigraphs to realize them as was done here. We have shown that realizability is equivalent to the existence of integer solutions to polynomial systems of equations, making realizability decidable but no less obscure. We emphasize that both the characterization and existence questions concern simple undirected graphs too as soon as information about their line graph is removed. That is, simple graphs do not relate to their simple cycles of length $\ell \neq 2$ any better than multidigraphs do to their simple cycles of any length.  

All of this demonstrates that walks are in fact much less dependent on the digraphs on which they take place than might have first been thought; and that a {``\emph{theory of walks}''} distinct from graph theory needs to be developed. Here we proceeded by relying on hike monoids, which provide a representation for walks and walk-like objects that is markedly detached from graphs sustaining those walks. As monoids, hike monoids are plain trace monoids. Unfortunately, while the latter
are well understood, there is no simple way to know which trace monoids are hike monoids. Although we have formulated most of our results in terms of dependency graphs, expressing them directly as statements on trace monoids makes this fact even more clear. Consider for example the
 following family of monoids with  identical independence relations but differing number of generators:
\begin{align*}
        T_0 &= \set{a,b,c,d\,|\, ac = ca,\, bd = db},\\
        T_n &= \set{a,b,c,d,x_1,\dots, x_n\,|\, ac = ca,\, bd = db}.
\end{align*}
Among these very similar trace monoids, we know that $T_0$, $T_1$, $T_3$, $T_5$, $T_9$, $T_{17}$ and $T_{29}$ are not hike monoids, while all others are.
Trace theory does not seem to be any better equipped than graph theory to address the questions raised here.
\\

\section*{Acknowledgments}
T. K. acknowledges postdoctoral funding by the Universit\'e du Littoral C\^ote d'Opale. This work is supported by the French National Research Agency (ANR) through the ANR-19-CE40-0006 project Algebraic Combinatorics of Hikes On Lattices (Alcohol).


\appendix
\section{Nine unrealizable graphs on 7 vertices}\label{AppA}

Consider the following nine graphs on seven vertices

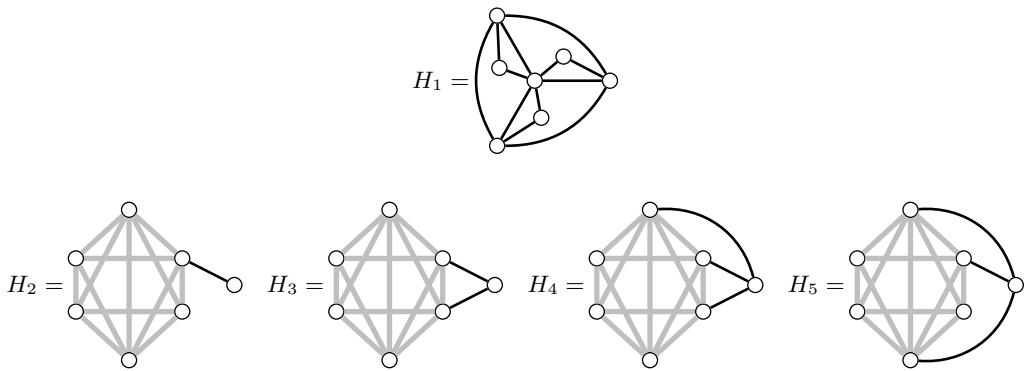
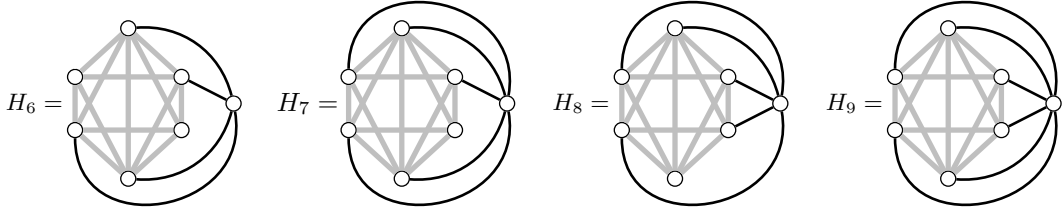
\begin{equation*}
\begin{tikzpicture}[Centering,line width=1]
\def\radi{1}
\def\radii{0.5}
\def\alpha{40}
\node[WhiteNode](c) at (0,0){};
\node[WhiteNode](b1) at (\alpha:\radii){};
\node[WhiteNode](b2) at (120+\alpha:\radii){};
\node[WhiteNode](b3) at (240+\alpha:\radii){};
\node[WhiteNode](a1) at (0:\radi){};
\node[WhiteNode](a2) at (120:\radi){};
\node[WhiteNode](a3) at (240:\radi){};
\node[NullNode](l) at (-1.25,0){\small ${H_1}=$};
\draw[-](c) to (a1);
\draw[-](c) to (a2);
\draw[-](c) to (a3);
\draw[-](c) to (b1);
\draw[-](c) to (b2);
\draw[-](c) to (b3);
\draw[-](a1) to (b1);
\draw[-](a2) to (b2);
\draw[-](a3) to (b3);
\path[-](a1) edge[bend right=30] (a2);
\path[-](a2) edge[bend right=30] (a3);
\path[-](a3) edge[bend right=30] (a1);
\end{tikzpicture}
\end{equation*}

\begin{equation*}
\begin{tikzpicture}[line width=1,Centering]
\def\radi{1}
\def\radii{0.5*\radi}
\def\alpha{45}
\def\height{1}
\node[WhiteNode](a) at ({\radi*cos(\alpha)},{\radii*sin(\alpha)}){};
\node[WhiteNode](b) at ({\radi*cos(\alpha+90)},{\radii*sin(\alpha+90)}){};circle(0.1);
\node[WhiteNode](c) at ({\radi*cos(\alpha+180)},{\radii*sin(\alpha+180)}){};circle(0.1);
\node[WhiteNode](d) at ({\radi*cos(\alpha+270)},{\radii*sin(\alpha+270)}){};circle(0.1);
\node[WhiteNode](e) at (0,\height){};
\node[WhiteNode](f) at (0,-\height){};
\node[WhiteNode](g) at (1.4,0){};
\draw[color=lightgray,line width=2](a) to (b);
\draw[color=lightgray,line width=2](b) to (c);
\draw[color=lightgray,line width=2](c) to (d);
\draw[color=lightgray,line width=2](d) to (a);
\draw[color=lightgray,line width=2](a) to (e);
\draw[color=lightgray,line width=2](b) to (e);
\draw[color=lightgray,line width=2](c) to (e);
\draw[color=lightgray,line width=2](d) to (e);
\draw[color=lightgray,line width=2](a) to (f);
\draw[color=lightgray,line width=2](b) to (f);
\draw[color=lightgray,line width=2](c) to (f);
\draw[color=lightgray,line width=2](d) to (f);
\draw[color=lightgray,line width=2](e) to (f);
\draw(g) to (a);
\node[NullNode](l) at (-1.25,0){\small $\modifJ{E_2}{H_2}=$};
\end{tikzpicture}
\begin{tikzpicture}[line width=1,Centering]
\def\radi{1}
\def\radii{0.5*\radi}
\def\alpha{45}
\def\height{1}
\node[WhiteNode](a) at ({\radi*cos(\alpha)},{\radii*sin(\alpha)}){};
\node[WhiteNode](b) at ({\radi*cos(\alpha+90)},{\radii*sin(\alpha+90)}){};circle(0.1);
\node[WhiteNode](c) at ({\radi*cos(\alpha+180)},{\radii*sin(\alpha+180)}){};circle(0.1);
\node[WhiteNode](d) at ({\radi*cos(\alpha+270)},{\radii*sin(\alpha+270)}){};circle(0.1);
\node[WhiteNode](e) at (0,\height){};
\node[WhiteNode](f) at (0,-\height){};
\node[WhiteNode](g) at (1.4,0){};
\draw[color=lightgray,line width=2](a) to (b);
\draw[color=lightgray,line width=2](b) to (c);
\draw[color=lightgray,line width=2](c) to (d);
\draw[color=lightgray,line width=2](d) to (a);
\draw[color=lightgray,line width=2](a) to (e);
\draw[color=lightgray,line width=2](b) to (e);
\draw[color=lightgray,line width=2](c) to (e);
\draw[color=lightgray,line width=2](d) to (e);
\draw[color=lightgray,line width=2](a) to (f);
\draw[color=lightgray,line width=2](b) to (f);
\draw[color=lightgray,line width=2](c) to (f);
\draw[color=lightgray,line width=2](d) to (f);
\draw[color=lightgray,line width=2](e) to (f);
\draw(g) to (a);
\draw(g) to (d);
\node[NullNode](l) at (-1.25,0){\small$\modifJ{E_3}{H_3}=$};
\end{tikzpicture}
\begin{tikzpicture}[line width=1,Centering]
\def\radi{1}
\def\radii{0.5*\radi}
\def\alpha{45}
\def\height{1}
\node[WhiteNode](a) at ({\radi*cos(\alpha)},{\radii*sin(\alpha)}){};
\node[WhiteNode](b) at ({\radi*cos(\alpha+90)},{\radii*sin(\alpha+90)}){};circle(0.1);
\node[WhiteNode](c) at ({\radi*cos(\alpha+180)},{\radii*sin(\alpha+180)}){};circle(0.1);
\node[WhiteNode](d) at ({\radi*cos(\alpha+270)},{\radii*sin(\alpha+270)}){};circle(0.1);
\node[WhiteNode](e) at (0,\height){};
\node[WhiteNode](f) at (0,-\height){};
\node[WhiteNode](g) at (1.4,0){};
\draw[color=lightgray,line width=2](a) to (b);
\draw[color=lightgray,line width=2](b) to (c);
\draw[color=lightgray,line width=2](c) to (d);
\draw[color=lightgray,line width=2](d) to (a);
\draw[color=lightgray,line width=2](a) to (e);
\draw[color=lightgray,line width=2](b) to (e);
\draw[color=lightgray,line width=2](c) to (e);
\draw[color=lightgray,line width=2](d) to (e);
\draw[color=lightgray,line width=2](a) to (f);
\draw[color=lightgray,line width=2](b) to (f);
\draw[color=lightgray,line width=2](c) to (f);
\draw[color=lightgray,line width=2](d) to (f);
\draw[color=lightgray,line width=2](e) to (f);
\draw(g) to (a);
\draw(g) to (d);
\path(g) edge[bend right=40] (e);
\node[NullNode](l) at (-1.25,0){\small$\modifJ{E_4}{H_4}=$};
\end{tikzpicture}
\begin{tikzpicture}[line width=1,Centering]
\def\radi{1}
\def\radii{0.5*\radi}
\def\alpha{45}
\def\height{1}
\node[WhiteNode](a) at ({\radi*cos(\alpha)},{\radii*sin(\alpha)}){};
\node[WhiteNode](b) at ({\radi*cos(\alpha+90)},{\radii*sin(\alpha+90)}){};circle(0.1);
\node[WhiteNode](c) at ({\radi*cos(\alpha+180)},{\radii*sin(\alpha+180)}){};circle(0.1);
\node[WhiteNode](d) at ({\radi*cos(\alpha+270)},{\radii*sin(\alpha+270)}){};circle(0.1);
\node[WhiteNode](e) at (0,\height){};
\node[WhiteNode](f) at (0,-\height){};
\node[WhiteNode](g) at (1.4,0){};
\draw[color=lightgray,line width=2](a) to (b);
\draw[color=lightgray,line width=2](b) to (c);
\draw[color=lightgray,line width=2](c) to (d);
\draw[color=lightgray,line width=2](d) to (a);
\draw[color=lightgray,line width=2](a) to (e);
\draw[color=lightgray,line width=2](b) to (e);
\draw[color=lightgray,line width=2](c) to (e);
\draw[color=lightgray,line width=2](d) to (e);
\draw[color=lightgray,line width=2](a) to (f);
\draw[color=lightgray,line width=2](b) to (f);
\draw[color=lightgray,line width=2](c) to (f);
\draw[color=lightgray,line width=2](d) to (f);
\draw[color=lightgray,line width=2](e) to (f);
\draw(g) to (a);
\path(g) edge[bend left=40] (f);
\path(g) edge[bend right=40] (e);
\node[NullNode](l) at (-1.25,0){\small$\modifJ{E_5}{H_5}=$};
\end{tikzpicture}
\end{equation*}
\begin{equation*}
\begin{tikzpicture}[line width=1,Centering]
\def\radi{1}
\def\radii{0.5*\radi}
\def\alpha{45}
\def\height{1}
\node[WhiteNode](a) at ({\radi*cos(\alpha)},{\radii*sin(\alpha)}){};
\node[WhiteNode](b) at ({\radi*cos(\alpha+90)},{\radii*sin(\alpha+90)}){};circle(0.1);
\node[WhiteNode](c) at ({\radi*cos(\alpha+180)},{\radii*sin(\alpha+180)}){};circle(0.1);
\node[WhiteNode](d) at ({\radi*cos(\alpha+270)},{\radii*sin(\alpha+270)}){};circle(0.1);
\node[WhiteNode](e) at (0,\height){};
\node[WhiteNode](f) at (0,-\height){};
\node[WhiteNode](g) at (1.4,0){};
\draw[color=lightgray,line width=2](a) to (b);
\draw[color=lightgray,line width=2](b) to (c);
\draw[color=lightgray,line width=2](c) to (d);
\draw[color=lightgray,line width=2](d) to (a);
\draw[color=lightgray,line width=2](a) to (e);
\draw[color=lightgray,line width=2](b) to (e);
\draw[color=lightgray,line width=2](c) to (e);
\draw[color=lightgray,line width=2](d) to (e);
\draw[color=lightgray,line width=2](a) to (f);
\draw[color=lightgray,line width=2](b) to (f);
\draw[color=lightgray,line width=2](c) to (f);
\draw[color=lightgray,line width=2](d) to (f);
\draw[color=lightgray,line width=2](e) to (f);
\draw(g) to (a);
\path(g) edge[bend left=40] (f);
\path(g) edge[bend right=40] (e);
\path(g) edge[out=280,in=270,distance=40] (c);
\path[color=white](g) edge[out=80,in=90,distance=40] (b);
\node[NullNode](l) at (-1.25,0){\small$\modifJ{E_6}{H_6}=$};
\end{tikzpicture}
\begin{tikzpicture}[line width=1,Centering]
\def\radi{1}
\def\radii{0.5*\radi}
\def\alpha{45}
\def\height{1}
\node[WhiteNode](a) at ({\radi*cos(\alpha)},{\radii*sin(\alpha)}){};
\node[WhiteNode](b) at ({\radi*cos(\alpha+90)},{\radii*sin(\alpha+90)}){};circle(0.1);
\node[WhiteNode](c) at ({\radi*cos(\alpha+180)},{\radii*sin(\alpha+180)}){};circle(0.1);
\node[WhiteNode](d) at ({\radi*cos(\alpha+270)},{\radii*sin(\alpha+270)}){};circle(0.1);
\node[WhiteNode](e) at (0,\height){};
\node[WhiteNode](f) at (0,-\height){};
\node[WhiteNode](g) at (1.4,0){};
\draw[color=lightgray,line width=2](a) to (b);
\draw[color=lightgray,line width=2](b) to (c);
\draw[color=lightgray,line width=2](c) to (d);
\draw[color=lightgray,line width=2](d) to (a);
\draw[color=lightgray,line width=2](a) to (e);
\draw[color=lightgray,line width=2](b) to (e);
\draw[color=lightgray,line width=2](c) to (e);
\draw[color=lightgray,line width=2](d) to (e);
\draw[color=lightgray,line width=2](a) to (f);
\draw[color=lightgray,line width=2](b) to (f);
\draw[color=lightgray,line width=2](c) to (f);
\draw[color=lightgray,line width=2](d) to (f);
\draw[color=lightgray,line width=2](e) to (f);
\draw(g) to (a);
\path(g) edge[bend left=40] (f);
\path(g) edge[bend right=40] (e);
\path(g) edge[out=280,in=270,distance=40] (c);
\path(g) edge[out=80,in=90,distance=40] (b);
\node[NullNode](l) at (-1.25,0){$\modifJ{E_7}{H_7}=$};
\end{tikzpicture}
\begin{tikzpicture}[line width=1,Centering]
\def\radi{1}
\def\radii{0.5*\radi}
\def\alpha{45}
\def\height{1}
\node[WhiteNode](a) at ({\radi*cos(\alpha)},{\radii*sin(\alpha)}){};
\node[WhiteNode](b) at ({\radi*cos(\alpha+90)},{\radii*sin(\alpha+90)}){};circle(0.1);
\node[WhiteNode](c) at ({\radi*cos(\alpha+180)},{\radii*sin(\alpha+180)}){};circle(0.1);
\node[WhiteNode](d) at ({\radi*cos(\alpha+270)},{\radii*sin(\alpha+270)}){};circle(0.1);
\node[WhiteNode](e) at (0,\height){};
\node[WhiteNode](f) at (0,-\height){};
\node[WhiteNode](g) at (1.4,0){};
\draw[color=lightgray,line width=2](a) to (b);
\draw[color=lightgray,line width=2](b) to (c);
\draw[color=lightgray,line width=2](c) to (d);
\draw[color=lightgray,line width=2](d) to (a);
\draw[color=lightgray,line width=2](a) to (e);
\draw[color=lightgray,line width=2](b) to (e);
\draw[color=lightgray,line width=2](c) to (e);
\draw[color=lightgray,line width=2](d) to (e);
\draw[color=lightgray,line width=2](a) to (f);
\draw[color=lightgray,line width=2](b) to (f);
\draw[color=lightgray,line width=2](c) to (f);
\draw[color=lightgray,line width=2](d) to (f);
\draw[color=lightgray,line width=2](e) to (f);
\draw(g) to (a);
\draw(g) to (d);
\path(g) edge[bend right=40] (e);
\path(g) edge[out=280,in=270,distance=40] (c);
\path(g) edge[out=80,in=90,distance=40] (b);
\node[NullNode](l) at (-1.25,0){\small$\modifJ{E_8}{H_8}=$};
\end{tikzpicture}
\begin{tikzpicture}[line width=1,Centering]
\def\radi{1}
\def\radii{0.5*\radi}
\def\alpha{45}
\def\height{1}
\node[WhiteNode](a) at ({\radi*cos(\alpha)},{\radii*sin(\alpha)}){};
\node[WhiteNode](b) at ({\radi*cos(\alpha+90)},{\radii*sin(\alpha+90)}){};circle(0.1);
\node[WhiteNode](c) at ({\radi*cos(\alpha+180)},{\radii*sin(\alpha+180)}){};circle(0.1);
\node[WhiteNode](d) at ({\radi*cos(\alpha+270)},{\radii*sin(\alpha+270)}){};circle(0.1);
\node[WhiteNode](e) at (0,\height){};
\node[WhiteNode](f) at (0,-\height){};
\node[WhiteNode](g) at (1.4,0){};
\draw[color=lightgray,line width=2](a) to (b);
\draw[color=lightgray,line width=2](b) to (c);
\draw[color=lightgray,line width=2](c) to (d);
\draw[color=lightgray,line width=2](d) to (a);
\draw[color=lightgray,line width=2](a) to (e);
\draw[color=lightgray,line width=2](b) to (e);
\draw[color=lightgray,line width=2](c) to (e);
\draw[color=lightgray,line width=2](d) to (e);
\draw[color=lightgray,line width=2](a) to (f);
\draw[color=lightgray,line width=2](b) to (f);
\draw[color=lightgray,line width=2](c) to (f);
\draw[color=lightgray,line width=2](d) to (f);
\draw[color=lightgray,line width=2](e) to (f);
\draw(g) to (a);
\draw(g) to (d);
\path(g) edge[bend right=40] (e);
\path(g) edge[bend left=40] (f);
\path(g) edge[out=280,in=270,distance=40] (c);
\path(g) edge[out=80,in=90,distance=40] (b);
\node[NullNode](l) at (-1.25,0){\small $\modifJ{E_9}{H_9}=$};
\end{tikzpicture}
\end{equation*}

System (\ref{surjectivity_pol_system}) admits no solution on minimal clique covers in the cases of \modifJ{$E_1,\,E_2,\,E_3,\, E_4,\,E_7,\, E_8$ and $E_9$}{$H_1,\,H_2,\,H_3,\, H_4,\,H_7,\, H_8$ and $H_9$}, while exhaustively checking for the absence 
of such a solution is not computationally feasible for $\modifJ{E_5}{H_5}$ and $\modifJ{E_6}{H_6}$. We present two ad-hoc arguments deciding the unrealizability of these graphs: the first for $\modifJ{E_1}{H_1}$ and the second for all other graphs.

\begin{description}
    \item[$\modifJ{E_1}{H_1}$] : Let us label $\modifJ{E_1}{H_1}$ as follows:
    
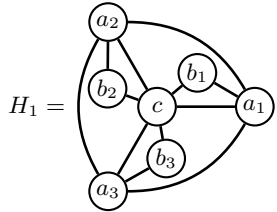
\begin{equation*}
        \begin{tikzpicture}[Centering,line width=1]
        \def\radi{1.3}
        \def\radii{0.7}
        \def\alpha{40}
        \node[WhiteLabelNode](c) at (0,0){\small$c$};
        \node[WhiteLabelNode](b1) at (\alpha:\radii){\small$b_1$};
        \node[WhiteLabelNode](b2) at (120+\alpha:\radii){\small$b_2$};
        \node[WhiteLabelNode](b3) at (240+\alpha:\radii){\small$b_3$};
        \node[WhiteLabelNode](a1) at (0:\radi){\small$a_1$};
        \node[WhiteLabelNode](a2) at (120:\radi){\small$a_2$};
        \node[WhiteLabelNode](a3) at (240:\radi){\small$a_3$};
        \node[NullNode](l) at (-1.6,0){\small $\modifJ{E_1}{H_1}=$};
        \draw[-](c) to (a1);
        \draw[-](c) to (a2);
        \draw[-](c) to (a3);
        \draw[-](c) to (b1);
        \draw[-](c) to (b2);
        \draw[-](c) to (b3);
        \draw[-](a1) to (b1);
        \draw[-](a2) to (b2);
        \draw[-](a3) to (b3);
        \path[-](a1) edge[bend right=30] (a2);
        \path[-](a2) edge[bend right=30] (a3);
        \path[-](a3) edge[bend right=30] (a1);
        \end{tikzpicture}
    \end{equation*}
     Suppose that $\modifJ{E_1}{H_1}$ is realizable and let $G$ be a digraph realizing it. Each triangle ($c\,a_1\,b_1,\, c\,a_1\,a_2, \ldots$) in $\modifJ{E_1}{H_1}$ corresponds to three simple cycles in $G$ sharing a common vertex: indeed, if there was no such 
        vertex then by Proposition~\ref{observations_prop_cycles} there would be two vertices in $\modifJ{E_1}{H_1}$ linked to every vertex of the triangle, which contradicts the structure of $\modifJ{E_1}{H_1}$.
    
    Suppose now that there is no vertex common to all 4 cycles $a_1$, $a_2$, $a_3$ and $c$. Denote by $v_1$ a vertex common to $a_1$, $a_2$ and $c$; by $v_2$ a vertex common to
    $a_2$, $a_3$ and $c$; and by $v_3$ a vertex common to $a_3$, $a_1$ and $c$. For $i\in\modifJ{[3]}{\{1,2,3\}}$, and $v,v'$ two vertices of $a_i$, denote by $v\rightarrow_iv'$ the path from $v$ to $v'$ 
    in $a_i$. Let us now consider cycles $v_1\rightarrow_2 v_2\rightarrow_3 v_3 \rightarrow_1 v_1$ and $v_3\rightarrow_3 v_2\rightarrow_2 v_1 \rightarrow_1 v_3$ and show that they 
    are simple. Suppose that the first of these two cycle is not simple and, without loss of generality, let $v$ be a vertex such that 
    $v_1\rightarrow_2 v_2 = v_1\rightarrow_2 v \rightarrow_2 v_2$, $v_2\rightarrow_3 v_3 = v_2\rightarrow_3 v \rightarrow_3 v_3$ and $v \rightarrow_2 v_2 \rightarrow_3 v$ is a simple 
    cycle. Then, by construction, this simple cycle can neither be any of $a_1$, $a_2$, $a_3$ and $c$ since there are missing vertices; nor can it be any of the $b_i$s since it is in contact with two $a_i$s which would
    contradicts the structure of $\modifJ{E_1}{H_1}$. Hence $v_1\rightarrow_2 v_2\rightarrow_3 v_3 \rightarrow_1 v_1$ is a simple cycle and, similarly, so is 
    $v_3\rightarrow_3 v_2\rightarrow_2 v_1 \rightarrow_1 v_3$. These two simple cycles can not be identifiable with one of the $a_i$s since else one of the vertices $v_i$ 
    would be common to $a_1$, $a_2$, $a_3$ and $c$. Neither can these two cycles be identifiable with one of the $b_i$s because they are in contact with more than two $a_i$s. Hence they must be distinct yet both be equal to $c$, a contradiction.
    
    We are thus forced to consider that there exists a vertex $v_0$ common to all 4 cycles $a_1$, $a_2$, $a_3$ and $c$. As earlier, denote by $v_1$ a vertex common to $a_1$, $b_1$ and $c$; by $v_2$ a vertex common 
    to $a_2$, $b_2$ and $c$; and by $v_3$ a vertex common to $a_3$, $b_3$ and $c$. Then $\{v_0,\,v_1,\,v_2,\,v_3\}$ corresponds to a minimal clique cover of $\modifJ{E_1}{H_1}$, implying that  $G$ would provide a solution to system (\ref{surjectivity_pol_system}) with a minimal clique cover which, we have checked, does not exist.
    
    \item[$\modifJ{E_2}{H_2}$ to $\modifJ{E_9}{H_9}$] : We know from Proposition~\ref{observations_prop_cycles} and the example preceding this proposition that if a realizable graph has a square as induced 
    subgraph then necessarily this square is part of an induced subgraph with the shape of $H$ on the left of \ref{EqSquare}. Furthermore, we also know that the simple cycles corresponding
    to the vertices of this subgraph must be organised as in Figure~\ref{F:SquareConstruction}.
    Now remark that adding an edge which is not a self-loop to Figure~\ref{F:SquareConstruction} always creates two simple cycles so that the corresponding image by $\phi$
    must have at least $8$ vertices. Consequently, the only way to realize a graph with $7$ vertices and a square as induced subgraph is to add a self-loop somewhere in Figure~\ref{F:SquareConstruction}.
    There are only two ways to do so which yield distinct hike monoids: either by adding the self-loop on a vertex that lies inside of one of the simple cycles $a,b,c,d$ and in no other cycle; or on a vertex that is in common to two cycles. The image by $\phi$ of the digraphs so-obtained are respectively:
    
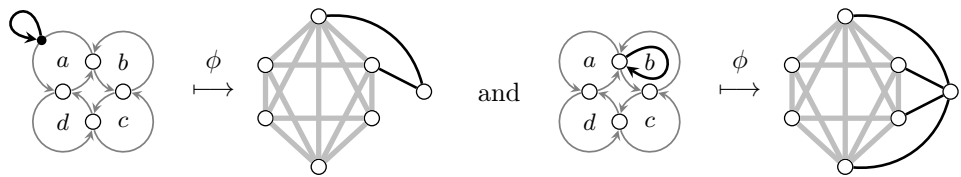
\begin{equation*}
\begin{tikzpicture}[line width=1,Centering]
\def\radi{0.4}
\def\alpha{16.5}
\node[WhiteNode](c1) at (90:\radi){};
\node[WhiteNode](c2) at (180:\radi){};
\node[WhiteNode](c3) at (270:\radi){};
\node[WhiteNode](c4) at (0:\radi){};
\node[NullNode](oc1) at (-\radi,\radi){\small$a$};
\node[NullNode](oc2) at (\radi,\radi){\small$b$};
\node[NullNode](oc3) at (\radi,-\radi){\small$c$};
\node[NullNode](oc4) at (-\radi,-\radi){\small$d$};
\draw[->,gray,line width=0.65](-\radi,\radi)+(\alpha:\radi)  arc (\alpha:270-\alpha:\radi);
\draw[->,gray,line width=0.65] (-\radi,\radi)+(270+\alpha:\radi) arc (-90+\alpha:-\alpha:\radi);
\draw[->,gray,line width=0.65](\radi,\radi)+(270+\alpha:\radi) arc (270+\alpha:540-\alpha:\radi);
\draw[->,gray,line width=0.65](\radi,\radi)+(180+\alpha:\radi) arc (180+\alpha:270-\alpha:\radi);
\draw[->,gray,gray,line width=0.65](\radi,-\radi)+(180+\alpha:\radi) arc (180+\alpha:450-\alpha:\radi);
\draw[->,gray,line width=0.65](\radi,-\radi)+(90+\alpha:\radi) arc (90+\alpha:180-\alpha:\radi);
\draw[->,gray,line width=0.65](-\radi,-\radi)+(90+\alpha:\radi) arc (90+\alpha:360-\alpha:\radi);
\draw[->,gray,line width=0.65](-\radi,-\radi)+(\alpha:\radi) arc (\alpha:90-\alpha:\radi);
\node[SmallBlackNode](e) at ($(-\radi,\radi)+(135:\radi)$){};
\path[->](e) edge[out=100,in=170,distance=20] (e);

\begin{scope}[shift={(3,0)}]
\def\radi{1}
\def\radii{0.5*\radi}
\def\alpha{45}
\def\height{1}
\node[WhiteNode](a) at ({\radi*cos(\alpha)},{\radii*sin(\alpha)}){};
\node[WhiteNode](b) at ({\radi*cos(\alpha+90)},{\radii*sin(\alpha+90)}){};circle(0.1);
\node[WhiteNode](c) at ({\radi*cos(\alpha+180)},{\radii*sin(\alpha+180)}){};circle(0.1);
\node[WhiteNode](d) at ({\radi*cos(\alpha+270)},{\radii*sin(\alpha+270)}){};circle(0.1);
\node[WhiteNode](e) at (0,\height){};
\node[WhiteNode](f) at (0,-\height){};
\node[WhiteNode](g) at (1.4,0){};
\draw[color=lightgray,line width=2](a) to (b);
\draw[color=lightgray,line width=2](b) to (c);
\draw[color=lightgray,line width=2](c) to (d);
\draw[color=lightgray,line width=2](d) to (a);
\draw[color=lightgray,line width=2](a) to (e);
\draw[color=lightgray,line width=2](b) to (e);
\draw[color=lightgray,line width=2](c) to (e);
\draw[color=lightgray,line width=2](d) to (e);
\draw[color=lightgray,line width=2](a) to (f);
\draw[color=lightgray,line width=2](b) to (f);
\draw[color=lightgray,line width=2](c) to (f);
\draw[color=lightgray,line width=2](d) to (f);
\draw[color=lightgray,line width=2](e) to (f);
\draw(g) to (a);
\path(g) edge[bend right=40] (e);
\end{scope}
\node[NullNode](l1) at (1.6,0){$\longmapsto$};
\node[NullNode](l2) at (1.6,0.4){$\phi$};
\begin{scope}[shift={(7,0)}]
 \def\radi{0.4}
\def\alpha{16.5}
\node[WhiteNode](c1) at (90:\radi){};
\node[WhiteNode](c2) at (180:\radi){};
\node[WhiteNode](c3) at (270:\radi){};
\node[WhiteNode](c4) at (0:\radi){};
\node[NullNode](oc1) at (-\radi,\radi){\small$a$};
\node[NullNode](oc2) at (\radi,\radi){\small$b$};
\node[NullNode](oc3) at (\radi,-\radi){\small$c$};
\node[NullNode](oc4) at (-\radi,-\radi){\small$d$};
\draw[->,gray,line width=0.65](-\radi,\radi)+(\alpha:\radi)  arc (\alpha:270-\alpha:\radi);
\draw[->,gray,line width=0.65] (-\radi,\radi)+(270+\alpha:\radi) arc (-90+\alpha:-\alpha:\radi);
\draw[->,gray,line width=0.65](\radi,\radi)+(270+\alpha:\radi) arc (270+\alpha:540-\alpha:\radi);
\draw[->,gray,line width=0.65](\radi,\radi)+(180+\alpha:\radi) arc (180+\alpha:270-\alpha:\radi);
\draw[->,gray,gray,line width=0.65](\radi,-\radi)+(180+\alpha:\radi) arc (180+\alpha:450-\alpha:\radi);
\draw[->,gray,line width=0.65](\radi,-\radi)+(90+\alpha:\radi) arc (90+\alpha:180-\alpha:\radi);
\draw[->,gray,line width=0.65](-\radi,-\radi)+(90+\alpha:\radi) arc (90+\alpha:360-\alpha:\radi);
\draw[->,gray,line width=0.65](-\radi,-\radi)+(\alpha:\radi) arc (\alpha:90-\alpha:\radi);
\path[->](c1) edge[out=35,in=-35,distance=25] (c1);
\end{scope}
\node[NullNode](l3) at (8.6,0){$\longmapsto$};
\node[NullNode](l4) at (8.6,0.4){$\phi$};
\begin{scope}[shift={(10,0)}]
\def\radi{1}
\def\radii{0.5*\radi}
\def\alpha{45}
\def\height{1}
\node[WhiteNode](a) at ({\radi*cos(\alpha)},{\radii*sin(\alpha)}){};
\node[WhiteNode](b) at ({\radi*cos(\alpha+90)},{\radii*sin(\alpha+90)}){};circle(0.1);
\node[WhiteNode](c) at ({\radi*cos(\alpha+180)},{\radii*sin(\alpha+180)}){};circle(0.1);
\node[WhiteNode](d) at ({\radi*cos(\alpha+270)},{\radii*sin(\alpha+270)}){};circle(0.1);
\node[WhiteNode](e) at (0,\height){};
\node[WhiteNode](f) at (0,-\height){};
\node[WhiteNode](g) at (1.4,0){};
\draw[color=lightgray,line width=2](a) to (b);
\draw[color=lightgray,line width=2](b) to (c);
\draw[color=lightgray,line width=2](c) to (d);
\draw[color=lightgray,line width=2](d) to (a);
\draw[color=lightgray,line width=2](a) to (e);
\draw[color=lightgray,line width=2](b) to (e);
\draw[color=lightgray,line width=2](c) to (e);
\draw[color=lightgray,line width=2](d) to (e);
\draw[color=lightgray,line width=2](a) to (f);
\draw[color=lightgray,line width=2](b) to (f);
\draw[color=lightgray,line width=2](c) to (f);
\draw[color=lightgray,line width=2](d) to (f);
\draw[color=lightgray,line width=2](e) to (f);
\draw(g) to (a);
\draw(g) to (d);
\path(g) edge[bend right=40] (e);
\path(g) edge[bend left=40] (f);
\end{scope}
\node[NullNode](l5) at (5.4,0){and};
\end{tikzpicture}
    \end{equation*}
    Since graphs $\modifJ{E_2}{H_2}$ to $\modifJ{E_9}{H_9}$ all have a square as induced subgraph but are different from the above two graphs, they are unrealizable.
\end{description}

\section*{References}
\bibliographystyle{elsarticle-num}

\end{document}